\documentclass[reqno]{amsart}
\usepackage[foot]{amsaddr}

\usepackage[a4paper,top=30mm,bottom=30mm,inner=30mm,outer=35mm]{geometry}

\pdfoutput=1

\usepackage{amssymb}
\usepackage{booktabs}
\usepackage{rotating}
\usepackage{color}
\usepackage{enumitem}
\usepackage[T1]{fontenc}
\usepackage[utf8]{inputenc}
\usepackage[l2tabu,orthodox]{nag}
\usepackage{tensor}
\usepackage[vcentermath]{youngtab}
\usepackage{nicefrac}
\usepackage{chemfig}
\usepackage{amsbsy}
\newcommand{\mbf}[1]{\pmb{#1}}

\usepackage{pifont}

\usepackage{dsfont}
\usepackage{mathbbol}

\usepackage{mathtools}

\usepackage{upgreek}
\usepackage{relsize}
\newcommand{\bigtau}{\ensuremath{\mathlarger{\mathsf{\uptau}}}}
\newcommand{\smalltau}{\ensuremath{\tau}}

\renewcommand{\div}{\ensuremath{\mathrm{div}}}
\DeclareMathOperator{\tr}{tr}

\usepackage[
		pdftex,
		bookmarks,
		bookmarksopen,
		bookmarksopenlevel=1,
		bookmarksnumbered,
		breaklinks,
		colorlinks,
		linkcolor=linkcolor,
		citecolor=citecolor,
		urlcolor=urlcolor,
	]{hyperref}

\hypersetup{
	pdftitle    = {Superintegrable systems on conformal surfaces},
	pdfauthor   = {Jonathan Kress, Konrad Schöbel, Andreas Vollmer},
	pdfsubject  = {article},
	pdfcreator  = {\LaTeX{} with `amsart' class},
	pdfproducer = {pdf\LaTeX},
	pdfkeywords = {superintegrable systems}
}

\definecolor{linkcolor}{rgb}{0.5,0.0,0.0}
\definecolor{citecolor}{rgb}{0.0,0.5,0.0}
\definecolor{urlcolor} {rgb}{0.0,0.0,0.5}

\title{Superintegrable systems on conformal surfaces}

\author{Jonathan Kress$^\sharp$}
\author{Konrad Schöbel$^*$}
\author{Andreas Vollmer$^{\sharp\flat\mathsection\star}$}

\email{j.kress@unsw.edu.au}
\email{konrad.schoebel@htwk-leipzig.de}
\email{andreas.vollmer@uni-hamburg.de} 
\email{andreas.d.vollmer@gmail.com}

\address[$\sharp$]{%
	School of Mathematics and Statistics \\
	University of New South Wales \\
	Sydney 2052 \\
	Australia
}

\address[$^*$]{%
	Faculty of Digital Transformation \\
	HTWK Leipzig University of Applied Sciences \\
	04251 Leipzig \\
	Germany
}

\address[$^\flat$]{%
	Institut f\"ur Geometrie and Topologie \\
	Universit\"at Stuttgart \\
	70049 Stuttgart \\
	Germany
}

\address[$^\mathsection$]{%
	Dipartimento di Scienze Matematiche
	"Giuseppe Luigi Lagrange" \\
	Politecnico di Torino \\
	Corso Duca degli Abruzzi, 24 \\
	10129 Torino \\
	Italy
}

\address[$^\star$]{%
	Fachbereich Mathematik
	Universit\"at Hamburg \\
	Bundesstra{\ss}e 55 \\
	20146 Hamburg \\
	Germany
}

\numberwithin{equation}{section}

\newtheorem{theorem}{Theorem}[section]
\newtheorem{proposition}[theorem]{Proposition}
\newtheorem{lemma}[theorem]{Lemma}

\theoremstyle{definition}
\newtheorem{definition}[theorem]{Definition}
\theoremstyle{remark}
\newtheorem{remark}[theorem]{Remark}
\newtheorem{example}[theorem]{Example}

\newcommand{\CC}{\ensuremath{\mathds{C}}}
\newcommand{\RR}{\ensuremath{\mathds{R}}}
\newcommand{\del}{\ensuremath{\partial}}
\newcommand{\diff}[2]{\ensuremath{\frac{\del #1}{\del #2}}}

\setlist[enumerate,1]{label=(\roman*)}

\begin{document}

\begin{abstract}
	We reconsider non-degenerate second order superintegrable systems in dimension two as geometric structures on conformal surfaces.
	This extends a formalism developed by the authors, initially introduced for (pseudo-)Riemannian manifolds of dimension three and higher.
	The governing equations of non-degenerate second order superintegrability in dimension two are structurally significantly different from those valid in higher dimensions.
	Specifically, we find conformally covariant structural equations, allowing one to classify the (conformal classes of) non-degenerate second order superintegrable systems on conformal surfaces geometrically.
	We then specialise to second order properly superintegrable systems on surfaces with a (pseudo-)Rie\-man\-nian metric and obtain structural equations in accordance with the known equations for Euclidean space. We finally give a single explicit set of purely algebraic equations defining the variety parametrising such systems on all constant curvature surfaces.
\end{abstract}

\maketitle
\tableofcontents

\section{Introduction}
Second order (maximally) superintegrable systems are Hamiltonian systems that, informally, possess a high degree of symmetry.
On surfaces, such systems have been intensively studied and, in particular, a classification of these systems exists for the complex case, 
\cite{KKM05a,Kalnins&Kress&Pogosyan&Miller}, as a list of normal forms up to isometries.
In this sense, second-order superintegrable systems are completely classified in dimension~2, including both properly and conformally superintegrable ones and 
comprising non-degenerate as well as degenerate systems \cite{KKM05a,KKM05b,Kalnins&Kress&Pogosyan&Miller}. A collection of many of the relevant works can be found in \cite{KKM18}.

For second-order properly superintegrable systems on Euclidean 2-space, the classification space has been shown to be an algebraic variety \cite{Kalnins&Kress&Miller,Kress&Schoebel}. 
Any 2-dimensional second-order conformally superintegrable system is St\"ackel equivalent to one of these flat models, or to the so-called ``generic system'' on the 2-sphere~\cite[Theorem 3]{KKM05b}.
St\"ackel transformations\footnote{%
	In the literature, the name \emph{coupling constant metamorphism} is also common, although these transformations are not the same in general and originated from different contexts. However, the concepts coincide in the context of interest here, see \cite{Post2010} for instance.}
are a way of identifying second order superintegrable systems on certain conformally equivalent metrics. General conformal transformations, however, do not preserve proper superintegrability, but lead to the more general concept of conformally superintegrable systems. A classification of the conformal classes (also known as St\"ackel classes) of \emph{non-degenerate} second order conformally superintegrable systems has been obtained in~\cite{Kress07,KKM05b}, see also~\cite{Vollmer21}.
Non-degeneracy, loosely speaking, here means that a maximal linear family of potentials is compatible with the integrals of motion.

Conformal classes have only been classified for non-degenerate systems. Yet, for degenerate second order superintegrable systems on surfaces, it has been proven \cite{KKMP09,KKM05a} that they always arise as the restriction of non-degenerate systems, emphasising the significance of non-degenerate systems in the 2-dimensional case. 

Second-order superintegrable systems have interesting interrelations to other
fields, as they give rise to quadratic Poisson algebras
\cite{Post11,BDK93,DT07,DT08,CMZ19} and corresponding representations as
differential or difference operators and hence to special functions.
Quadratic algebras have been used to classify superintegrable systems up to
St\"ackel equivalence \cite{Kress07}.
All non-degenerate two-dimensional systems arise from a ``generic'' system on the $2$-sphere via B\^ocher contractions, which are closely linked to a generalisation of \.{I}n\"on\"u-Wigner Lie algebra contractions \cite{KMP13,RKMS2017}.
This is particularly intriguing in the light of a correspondence that has been
observed in \cite{KMP13} between the hierarchy of isometry classes of
non-degenerate superintegrable systems in dimension two on one hand and the
Askey-Wilson scheme of hypergeometric orthogonal polynomials and their
degenerations on the other \cite{Askey&Wilson}.  We expect that the present
work will provide the groundwork to reveal an algebraic variety underlying the
Askey-Wilson scheme, which would allow to study hypergeometric orthogonal
polynomials via algebraic geometric tools.

\subsection{Novelty of the present paper}
Recently developed methods~\cite{KSV2023,KSV2024} are revealing geometric structures underlying superintegrable systems in dimensions $n\geq3$.  In the light of these advances, we reassess second order systems in two dimensions, approaching the classification problem for superintegrable systems in terms of algebraic varieties endowed with a natural action of the isometry group.
Such an algebraic geometric classification exists for systems on flat surfaces, but is not known for any surface with curvature, not even for the 2-sphere, although this case is studied in \cite{Kalnins&Kress&Pogosyan&Miller}. 
Describing the classification space of second-order superintegrable systems as an algebraic variety is one goal of the present paper, providing the basis for a holistic understanding of superintegrable systems.

Our novel approach consistently adopts a conformal viewpoint.
The aim of the present paper is to reconsider conformal classes of (second-order maximally) conformally superintegrable systems as geometric structures on conformal surfaces and to offer a formulation of such systems on Riemann surfaces.
We thereby shed new light onto the subject, obtaining a tensorial formulation, and thus geometric understanding, of the objects and computations in \cite{Kress&Schoebel,Kalnins&Kress&Miller,KKM07b,KKM07c}.
In the spirit of \cite{KSV2024}, we study the integrability conditions for the compatible potentials and (conformal) Killing tensors of a (conformally) superintegrable system. Specifying the values of certain scalar functions at a point suffices to reconstruct the system locally. The reconstruction of a properly superintegrable system on a surface with a specific Riemannian metric, however, requires certain algebraic obstruction conditions to be satisfied.

\subsection{Structure of the paper}\label{sec:structure}

The paper is organised as follows: After a brief summary of the notation, we recall the link between Riemann surfaces and conformal surfaces in Section~\ref{sec:Riemann.surface.conformal} as well as the definition and basic properties of conformally superintegrable systems in Sections~\ref{sec:conformal.systems} and~\ref{sec:BD}. In Section~\ref{sec:2D.systems} we discuss the properties of non-degenerate second-order (maximally) conformally superintegrable systems, reconsidering the results from~\cite{KSV2023,KSV2024}. This discussion will be carried out in tensorial language.
We obtain the integrability conditions for the compatible potentials and conformal Killing tensors in a general form.
These integrability conditions are then recast, in Section~\ref{sec:local.conf.equations}, using local isothermal coordinates.
Working simultaneously on a tensorial and a coordinate level will turn out to be useful since in low dimension either formulation has its advantages and drawbacks.
Conformal transformations of non-degenerate superintegrable systems will briefly be recalled in Section~\ref{sec:conformal.trafos}.

Section~\ref{sec:riemann.surface} is dedicated to conformally superintegrable systems, which are studied using specific \emph{gauge choices}. By this we mean that we use the freedom to perform conformal rescalings in order to impose certain conditions that facilitate the computations. Namely, we use the so-called \emph{standard gauge}, in which the primary structure tensor is trace-free (see Section~\ref{sec:standard.conformal}), and the \emph{flat gauge}, in which the Riemannian manifold is flat (see Section~\ref{sec:flat.conformal}). The general (gauge-independent) formulas are then discussed in Section~\ref{sec:general.conformal}. In particular, an explicit PDE system is obtained for certain tensor fields -- called (conformal) superintegrable structure tensors -- that encode the system.

Section~\ref{sec:proper} is dedicated to properly superintegrable systems, which are briefly reviewed in Sections~\ref{sec:proper.preliminaries} and~\ref{sec:proper.non-degeneracy}. The integrability conditions for the potentials and (proper) Killing tensors of non-degenerate properly superintegrable systems are discussed in Sections~\ref{sec:proper.integrability}. In contrast to the conformal case, the initial data for the explicit PDE system for the (proper) superintegrable structure tensors are subject to algebraic conditions.
This is consistent with the special cases appearing in the literature, in particular the flat case. We detail the specific correspondences between our framework and the existing formulations in Section~\ref{sec:euclidean}.
The 2-sphere is briefly discussed in Section~\ref{sec:sphere}.

\subsection{Notation}\label{sec:notation}

As here is not the place for a detailed introduction to representation theory and Young tableaux, we refer the interested reader to the literature, solely citing some conventions to be used in what follows.
Most importantly, we are going to denote tensor symmetries by Young tableaux. For instance,
\[
	{\young(ij)}_\circ T_{ijk} = T_{ijk}+T_{jik}-\frac2n g_{ij} T\indices{^a_{ak}}
\]
is a projection of the tensor $T_{ijk}$ which is symmetric and tracefree in $(i,j)$.
More complicated tableaux denote a complete antisymmetrisations in the indices of each column, followed by a complete symmetrisation in the indices of each row. For example
$$ {\young(ji,k)}=\young(ij)\young(j,k). $$
The respective dual operator, where symmetrisations are performed before antisymmetrisations, is indicated by an asterisk.
For instance, 
$$ {\young(ji,k)}^*=\young(j,k)\young(ij). $$
Consequently, the projection of a tensor $T_{ijkl}$ onto its component with algebraic curvature symmetry reads
\[
	{\young(ij,kl)}^* T_{ijkl} = \young(i,k)\young(j,l)\young(ij)\young(kl) 
	T_{ijkl}\,.
\]
Concerning differentiation, we indicate usual (i.e., partial) derivatives by a semi-colon. In local coordinates $(x,y)$, for instance, we will have
\[
	\nabla V = V_{;x}\partial_x + V_{;y}\partial_y.
\]
Covariant derivatives, on the other hand, shall be denoted by a comma.

\subsection*{Acknowledgements.}
We would like to thank the contributors of the computer algebra systems \texttt{cadabra2} \cite{Peeters06,Peeters07} and \texttt{SageMath} \cite{sagemath}, which have been used to find, prove and simplify some of the most important results in this work, for providing, maintaining and extending their software and distributing it under a free license.

This work was funded by the German Research Foundation (DFG), project 540196982, and initiated with funding by DFG under the project 353063958. AV also acknowledges support from the project PRIN 2017 ``Real and Complex Manifolds: Topology, Geometry and holomorphic dynamics'' and from the Forschungsfonds of the University of Hamburg.

\section{Preliminaries}
\label{sec:preliminaries}

The goal for the first half of the paper is to reconsider second order conformally superintegrable systems in dimension~2 from the viewpoint of Riemann surfaces. This is possible since the complex unit provides a concept of orthogonality, establishing a 1-to-1 correspondence between conformal classes of Riemannian metrics and Riemann surfaces.

\subsection{Riemann surfaces as conformal surfaces}\label{sec:Riemann.surface.conformal}

We briefly introduce the correspondence between Riemann surfaces and conformal classes of Riemannian manifolds.
Two Riemannian metrics $g_1,g_2$ are said to be \emph{conformally equivalent}, $g_2\sim g_1$, if there is a real analytic function $\Omega:S\to\RR$ such that $g_2=\Omega^2g_1$.
A conformal structure then is an equivalence class of Riemannian metrics under this equivalence.
It has been shown in \cite{KSV2024} how conformally superintegrable systems can be identified under conformal rescalings.

A Riemann surface $S\subset\CC$ can, equivalently, be thought of as an oriented manifold of real dimension $n=2$
endowed with a conformal structure. 
Indeed, locally consider a patch $S\subset\CC$ with the standard complex structure $i$. The class of compatible metrics on $S$ associated with $i$ is then given by $[g]=\{\phi(z,\bar z)^2\,dzd\bar z:\phi\ne0\}$. A conformal class of Riemannian metrics, on the other hand, induces a complex structure on the surface $S$ via the existence of Gau{\ss}' isothermal coordinates and orthogonality \cite[Ch.~3.11]{Jost2006}.
Here, we aim to reformulate the formalism developed in~\cite{KSV2024} such that a concrete (pseudo-)Riemannian metric needs not be specified. The aforementioned correspondence hence allows us to use the complex structure $J=i$ instead of the conformal class $[g]$.
One of the advantages of such a formulation is that we are going to be able to describe the superintegrable system concisely using two complex functions, which satisfy certain integrability conditions.

Locally, we may consider $S\subset\CC$ with the standard complex structure $J=i$.
Choosing a specific Riemannian metric $g=\phi^2dzd\bar z$ from the class $[g]$ fixes the scalar function $\phi$ where we use Gau{\ss}' isothermal coordinates, such that $g$ is brought into the form ($z=x+iy$, $w=x-iy$)
\begin{equation}\label{eqn:local.coordinates}
	g = \phi^2(x,y)(dx^2+dy^2) = \phi^2(z,w)\,dzdw\,,
\end{equation}
with a real-valued function $\phi=\phi(x,y)=\phi(z,w)$ where $w=\bar z$.
By abuse of notation we use the symbol $\phi$ in both charts.
Although we focus on Riemannian metrics, we will sometimes comment on surfaces with a pseudo-Riemannian metric,
\begin{equation}\label{eqn:local.coordinates.pseudo}
	g = \phi^2(x,y)(dx^2-dy^2) = \phi^2(z,w)\,dzdw\,,
\end{equation}
where $z=x+y$ and $w=x-y$ are real coordinates.
A reader only interested in Riemannian metrics may therefore simply identify $w=\bar z$ in the following.
In Section~\ref{sec:riemann.surface} we focus on systems on Riemann surfaces and therefore set $w=\bar z$ explicitly.
In the literature, e.g.~\cite{Kalnins&Kress&Miller,KKM07b,KKM07c}, $z$ and $w$ are often taken to be complex coordinates,
\begin{equation}
	g = \phi^2(z,w)\,dzdw\,,
\end{equation}
where also $\phi$ is allowed to be complex valued. Our notation of $\bar z$ as $w$ may help the reader interested  in pursuing this viewpoint to generalise the equations presented here to complex coordinates.

\subsection{Conformally superintegrable systems}\label{sec:conformal.systems}
We introduce second-order (maximally) conformally superintegrable systems, limiting ourselves to the case of dimension $n=2$ that is of interest here.
Hence, let $M$ be a smooth manifold of dimension $n=2$ with metric $g$. The cotangent space $T^*M$ can, from a physical viewpoint, be interpreted as the \emph{phase space}: Darboux coordinates $(\mbf q,\mbf p)$ are thus called \emph{positions} $\mbf q=(q^1,q^2)$ and \emph{momenta} $\mbf p=(p_1,p_2)$. Note that a diffeomorphism $\varphi$ of $M$ induces a (fibre-preserving) symplectomorphism on $T^*M$ and that any symplectomorphism of $T^*M$ that preserves the tautological 1-form is fibre-preserving.

A \emph{Hamiltonian system} is a dynamical system characterised by a Hamiltonian $H(\mathbf p,\mathbf q)$, i.e.~a function on phase space. Its temporal evolution is governed by the equations of motion:
\begin{align*}
	\dot{\mathbf p}&=-\frac{\partial H}{\partial\mathbf q}&
	\dot{\mathbf q}&=+\frac{\partial H}{\partial\mathbf p}
\end{align*}
A function $F(\mathbf p,\mathbf q)$ on the phase space is called a
\emph{constant of motion} or \emph{first integral}, if it is constant under
this evolution, i.e.\ if
\[
	\dot F
	=\frac{\partial F}{\partial\mathbf q}\dot{\mathbf q}
	+\frac{\partial F}{\partial\mathbf p}\dot{\mathbf p}
	=\frac{\partial F}{\partial\mathbf q}\frac{\partial H}{\partial\mathbf p}
	-\frac{\partial F}{\partial\mathbf p}\frac{\partial H}{\partial\mathbf q}
	=0
\]
or
\[
	\{F,H\}=0,
\]
where
\[
	\{F,G\}=
	\sum_{i=1}^n
	\left(
		\frac{\partial F}{\partial q^i}
		\frac{\partial G}{\partial p_i}
		-
		\frac{\partial G}{\partial q^i}
		\frac{\partial F}{\partial p_i}
	\right)
	=
	\frac{\partial F}{\partial q^1}
	\frac{\partial G}{\partial p_1}
	+\frac{\partial F}{\partial q^2}
	\frac{\partial G}{\partial p_2}
	-\frac{\partial G}{\partial q^1}
	\frac{\partial F}{\partial p_1}
	-\frac{\partial G}{\partial q^2}
	\frac{\partial F}{\partial p_2}
\]
is the canonical Poisson bracket.  Such a constant of motion restricts the
trajectory of the system to a hypersurface in phase space.  If the system
possesses the maximal number of $2n-1=3$ functionally independent constants of
motion $F^{(\alpha)}$, i.e.\ $F^{(1)}$, $F^{(2)}$ and $F^{(0)}=H$, then its trajectory in phase space is the
(unparametrised) curve given as the intersection of the hypersurfaces $F^{(\alpha)}(\mathbf p,\mathbf q)=c^{(\alpha)}$, where the constants $c^{(\alpha)}$ are determined by the initial conditions.
For such systems we can solve the equations of motion exactly and in a purely algebraic way, without having to solve explicitly any differential equation.
The metric $g$ on $M$ gives rise to a natural Hamiltonian $H:T^*M\to\RR$,
\begin{equation}\label{eq:Hamiltonian}
	H(\mbf q,\mbf p)=g^{-1}_{\mbf q}(\mbf p,\mbf p)+V(\mbf q)\,,
\end{equation}
where we will usually suppress $\mbf q$ in the notation, if there is no risk of confusion.

\begin{definition}\label{def:main.notions}
	~
	\begin{enumerate}
		\item
		By a \emph{conformally (maximally) superintegrable system} in dimension two, we mean a 
		Hamiltonian system admitting $2n-1=3$ functionally independent conformal integrals of 
		the motion $F^{(\alpha)}$,
		\begin{align}
			\label{eq:conformal.integral}
			\{F^{(\alpha)},H\}&= \rho^{(\alpha)}\,H &
			\alpha&\in\{0,1,2\}\,,
		\end{align}
		with functions $\rho^{(\alpha)}(\mathbf{p},\mathbf{q})$ polynomial in momenta.
		The Hamiltonian is required to be among the conformal integrals, and by convention
		\[
			F^{(0)} = H\,,\quad \rho^{(0)}=0.
		\]
		\item
		A conformal integral of the motion is \emph{second order} if it 
		is of the form
		\begin{equation}
			\label{eq:quadratic.conformal}
			F^{(\alpha)}=C^{(\alpha)}+V^{(\alpha)},
		\end{equation}
		where
		\[
		C^{(\alpha)}(\mathbf p,\mathbf 
		q)=\sum_{i=1}^nC^{(\alpha)ij}(\mathbf q)p_ip_j
		\]
		is quadratic in momenta and $V^{(\alpha)}=V^{(\alpha)}(\mathbf q)$
		a function depending only on positions.
		A conformally superintegrable system is \emph{second order} if its conformal integrals $F^{(\alpha)}$ are second order and if~\eqref{eq:Hamiltonian} is given by the Riemannian metric 
		$g_{ij}(\mathbf q)$ on the underlying manifold.
		\item
		We call $V$ a \emph{conformal superintegrable potential} if the 
		Hamiltonian~\eqref{eq:Hamiltonian} gives rise to a conformally superintegrable 
		system.
	\end{enumerate}
\end{definition}
\noindent In this article we are concerned exclusively with \emph{second order maximally} superintegrable systems and thus we typically omit the terms ``second order'' and ``maximally'' without further mentioning.

\begin{remark}
	Instead of the characterisation given in the formal definition, a conformally superintegrable system may be specified by the following, equivalent data: the underlying manifold $(M,g)$, a potential $V\in\mathcal C^\infty(M)$, and the space $\mathcal F$ of conformal integrals compatible with the potential $V$. In the following we allow a linear family of potentials, and hence a conformally superintegrable system may be viewed as a quadruple $(M,g,\mathcal V,\mathcal F)$ satisfying the conditions outlined earlier in this section for any choice of $V\in\mathcal V$.
\end{remark}

\subsection{Bertrand-Darboux condition}\label{sec:BD}

Now consider the condition \eqref{eq:conformal.integral} for \eqref{eq:quadratic.conformal} and \eqref{eq:Hamiltonian}. According to the degree in momenta, it splits into a cubic and a linear part (with respect to~$\mathbf p$):
\begin{subequations}
	\label{eq:1st+3rd}
	\begin{align}
		\label{eq:3rd}\{C^{(\alpha)},G\}&=\rho^{(\alpha)}G\\
		\label{eq:1st}\{C^{(\alpha)},V\}+\{V^{(\alpha)},G\}&=\rho^{(\alpha)} V
	\end{align}
	where $G(\mathbf q,\mathbf p)=g^{-1}_\mathbf{q}(\mathbf p,\mathbf p)$.
\end{subequations}
The condition~\eqref{eq:3rd} is equivalent to the requirement that $C_{ij}^{(\alpha)}$ are the components of a conformal Killing tensor in the following sense.
\begin{definition}
	A (second order) \emph{conformal Killing tensor} is a symmetric tensor 
	field on a pseudo-Riemannian manifold satisfying the Killing equation
	\begin{equation}
		\label{eq:conformal.Killing}
		{\young(ijk)}_\circ C_{ij,k} = 0\,.
	\end{equation}
\end{definition}
\noindent 
In particular, the metric $g$ itself is trivially a conformal Killing tensor.
Note that for a second-order conformal integral, \eqref{eq:conformal.integral} is a cubic polynomial in momenta and therefore $\rho^{(\alpha)}$ has to be linear in $\mathbf{p}$, i.e.~$\rho^{(\alpha)}=\rho_ap_bg^{ab}$.
The coefficients $\rho_a$ are components of the 1-form $\rho^{(\alpha)}$.
A direct computation confirms that
$$ \rho_k = \frac{2}{n+2}C\indices{^a_{k,a}} = \frac12C\indices{^a_{k,a}}\,, $$
where we drop the superscript $(\alpha)$ for brevity.
Note that given a conformal Killing tensor $C_{ij}$ and a smooth function $\lambda$ on $M$, $C_{ij}+\lambda g_{ij}$ also is a conformal Killing tensor of $g$.
In fact, observe that with $F=C^{ij}p_ip_j+W$ also $F-\lambda H$ is a conformal integral for $H$, for any function $\lambda$. Choosing $\lambda=C\indices{^a_a}$, we conclude that in Definition~\ref{def:main.notions} we may choose the conformal integrals $F^{(\alpha)}$ for $\alpha\ne0$ such that their quadratic part arises from a tracefree conformal Killing tensor.
\begin{remark}
	Without loss of generality, from now on we will tacitly suppose that the quadratic part of a conformal integral $F^{(\alpha)}$, for $\alpha\ne0$, is associated with a \emph{tracefree} conformal Killing tensor.
	We denote the space of tracefree conformal Killing tensors associated with $\mathcal F$ by $\mathring{\mathcal{C}}$.
\end{remark}

We continue our discussion of the equations~\eqref{eq:1st+3rd}.
Note that the metric $g$ allows us to identify symmetric forms and endomorphisms. Interpreting a conformal Killing tensor in this way as an endomorphism on 1-forms, equation \eqref{eq:1st} can be written in the form
\[
	dV^{(\alpha)}=C^{(\alpha)}dV-\rho V,
\]
and shows that, once the conformal Killing tensors $C^{(\alpha)}$ are known, the potentials $V^{(\alpha)}$ can be recovered from $V=V^{(0)}$ up to an irrelevant constant, provided the integrability condition
\begin{equation}\label{eq:dCdV}
	dC^{(\alpha)}dV - Vd\rho^{(\alpha)} - dV\wedge\rho^{(\alpha)}
	= d(C^{(\alpha)}dV-\rho^{(\alpha)} V) = 0
\end{equation}
holds, where $\wedge$ denotes the wedge operator in the exterior algebra of $T^*M$.
Equation~\eqref{eq:dCdV} is called the (conformal) \emph{Bertrand-Darboux condition}. Note that the potentials $V^{(\alpha)}$ for $\alpha\not=0$ are now eliminated from our equations.
As we will see in the following, the remaining potential $V=V^{(0)}$ can be eliminated as well, leaving equations on the conformal Killing tensors $C^{(\alpha)}$ alone.
In components, the Bertrand-Darboux condition \eqref{eq:dCdV} for a conformal Killing tensor $C=C^{(\alpha)}$ in a superintegrable system then reads, dropping the superscript $(\alpha)$,
\begin{equation}
	\label{eq:dCdV:ij}
	\young(i,j)
	\bigl(
		C\indices{^a_i}V_{,ja}+C\indices{^a_{i,j}}V_{,a}
		-V_{,i}\rho_{j}-V\rho_{i,j}
	\bigr)=0\,,
	\qquad\text{where}\quad \rho_k=\frac12\,C\indices{^a_{k,a}}.
\end{equation}

\subsection{Non-degenerate systems}\label{sec:2D.systems}
We now aim at a preliminary version of the integrability conditions for compatible potentials and (trace-free) conformal Killing tensors of non-degenerate systems. We obtain such conditions by reviewing~\cite{KSV2023,KSV2024}, which we adapt to the 2-dimensional case.
A second-order conformally superintegrable system is said to be \emph{irreducible} if the space $\mathring{\mathcal{C}}$ of tracefree conformal Killing tensors associated to it via the integrals of motion forms an irreducible set of endomorphisms.
Now consider~\eqref{eq:dCdV}. The following is shown in \cite{KSV2023,KSV2024}, see also \cite{KKM05a}: for an irreducible system, one can solve~\eqref{eq:dCdV} for all second derivatives of the potential $V$, except for the Laplacian $\Delta V$, and the solution can be written in terms of the tracefree conformal Killing tensors and their derivatives. One therefore obtains an expression of the form
\begin{equation}\label{eq:conformal.Wilczynski}
	V_{,ij} = T\indices{_{ij}^k}V_{,k} + \frac12\,g_{ij}\Delta V + \bigtau_{ij}V
\end{equation}
where the tensors $T_{ijk}$ and $\bigtau_{ij}$ only depend on the space spanned by the trace-free conformal Killing tensors $C^{(\alpha)}$.
In the present paper, our main interest is in non-degenerate systems.
\begin{definition}\label{def:non-degenerate.conformal}
	Consider the quadruple $(M,g,\mathcal V,\mathcal F)$ composed of a smooth manifold $M$ of dimension $n=2$ with metric $g$, a space $\mathcal V$ of scalar functions and a space of conformal integrals $\mathcal F$.
	We say that $(M,g,\mathcal V,\mathcal F)$ is \emph{non-degenerate} if
	\begin{itemize}
		\item the space $\mathcal F$ has dimension $2n-1=3$,
		\item the space $\mathcal V$ has dimension $n+2=4$, and
		\item for any $V\in\mathcal V$, and any $F\in\mathcal F$, Equations~\eqref{eq:conformal.Wilczynski} and~\eqref{eq:dCdV} are satisfied.
	\end{itemize}
	An irreducible second-order conformally superintegrable system is called \emph{non-degenerate} if such a non-degenerate quadruple exists, where~$\mathcal V$ and~$\mathcal F$ include, respectively, its potential and associated conformal integrals.
\end{definition}

\begin{remark}
	Note that this definition is consistent with the definition of non-degenerate second-order superintegrable systems on manifolds of dimension $n\geq3$, see \cite{KSV2023,KSV2024}.
	These references also introduce so-called \emph{abundant} systems:
	non-degenerate systems are called abundant, if their $(n+2)$-dimensional space of potentials is compatible, via~\eqref{eq:dCdV}, with a space of tracefree conformal Killing tensors of dimension \smash{$\frac{n(n+1)}{2}-1=\frac12(n-1)(n+2)$} \cite{KSV2023,KSV2024}.
	Note that this condition is always satisfied in dimension~2, as $2n-2=2=\frac12(n-1)(n+2)$, and hence the concepts of non-degeneracy and abundantness coincide.
\end{remark}

Because of the symmetries of $V_{,ij}$, we infer that the structure tensors of a non-degenerate system satisfy
\[
	\young(i,j)\,T\indices{_{ij}^k}=0\,,
	\qquad
	\young(i,j)\bigtau_{ij} = 0\,,
	\qquad
	T\indices{_a^{ak}} = 0\,,
	\qquad
	\bigtau\indices{^a_a} = 0\,.
\]
and that, moreover, the structure tensors are unique for such systems, c.f.~\cite{KSV2023,KSV2024}.
The following lemma is a special case of Proposition 4.1 in~\cite{KSV2024}.
\begin{lemma}\label{la:prolongation:V}
	A non-degenerate conformally superintegrable potential in dimension~$2$, 
	satisfying~\eqref{eq:conformal.Wilczynski}, forms the closed system
	\begin{subequations}\label{eqn:prolongation:V}
		\begin{alignat}{9}
			\label{eqn:prolongation:V:1}
			V_{,ij}
			&=T\indices{_{ij}^m}&&V_{,m}
			+&\bigtau_{ij}&V
			+&\tfrac{1}{2}g_{ij}&\Delta V \\
			\label{eqn:prolongation:V:2}
			(\Delta V)_{,k}
			&=2q\indices{_k^m}&&V_{,m}
			+&2\gamma_k&V
			+&t_k&\Delta V\,,
		\end{alignat}
	\end{subequations}
	where
	\[
	t_j = \frac12\,T\indices{_{aj}^a}\,,
	\qquad
	q\indices{_j^m} = Q\indices{_{ij}^{im}}\,,
	\qquad
	\gamma_k = \Gamma\indices{_{ak}^a}
	\]
	are derived from
	\begin{align}
		\label{eqn:Q}
		Q\indices{_{ijk}^m}
		&=
		T\indices{_{ij}^m_{,k}}
		+T\indices{_{ij}^l}T\indices{_{lk}^m}
		-R\indices{_{ijk}^m} + \bigtau_{ij}g_k^m
		\\
		\label{eq:Gamma}
		\Gamma_{ijk}
		&=
		\bigtau_{ij,k}+T\indices{_{ij}^a}\bigtau_{ak}
		\\
		\label{eq:Riemann}
		R_{ijkl}
		&=
		\frac{R}{2}\,( g_{ik}g_{jl}-g_{il}g_{jk} )
	\end{align}
	and where $R$ denotes the Gau{\ss} curvature.
\end{lemma}
\begin{proof}
	Take a covariant derivative of~\eqref{eq:conformal.Wilczynski} and 
	contract the resulting equation. Then solve this equation for $(\Delta V)_{,k}$.
\end{proof}

\begin{remark}
	A system of PDEs of the form~\eqref{eqn:prolongation:V}, in which derivatives of a certain order $m$ of the unknowns are expressed in terms of lower order derivatives of these unknowns, is called a \emph{prolongation system}. If all the derivatives of order $m$ can be expressed in such a way, the system is called a \emph{prolongation system of finite type} or a \emph{closed system}. Since in such a system the number of equations exceeds the number of unknown functions, it is \emph{overdetermined}.
	
	The prolongation system~\eqref{eqn:prolongation:V} is indeed an overdetermined prolongation system of finite type.
	As such, this system of PDEs can be integrated if the \emph{integrability conditions} are satisfied. This means that the Ricci identities are satisfied for each of the equations, i.e.~for~\eqref{eqn:prolongation:V:1} and~\eqref{eqn:prolongation:V:2}.
\end{remark}

\begin{remark}
	To avoid confusion, we would like to point out that the definition of $t_i$ in Lemma~\ref{la:prolongation:V} differs by a factor of $\nicefrac12$ from that in \cite{KSV2024}, where the 1-forms $t_i$ and $\bar t_i$ are used that only differ by a constant, dimension-dependent factor. The convention here is chosen as to make the resulting expressions as concise as possible.
\end{remark}

The following proposition follows from results in Section~5 of~\cite{KSV2024}.
\begin{proposition}\label{eq:conformal.K.eqns} 
	A non-degenerate second order conformally superintegrable system in 
	dimension~2 satisfies
	\begin{align}
		\label{eqn:prolongation.C}
		\nabla_kC_{ij} &= P\indices{_{ijk}^{mn}}C_{mn}\,,
		\\
		\intertext{and}
		\label{eqn:raw.tau}	
		\bigtau_{ij}
		&= \Lambda\indices{_{ij}^a_{,a}}
		-\Lambda\indices{^a_{ja,i}}
		+\Lambda_{iab}P\indices{^{abc}_{cj}}
		-\Lambda^{cab}P_{abijc}
	\end{align}
	where
	\begin{align*}
		P_{ijkmn} &= \frac16\,\young(mn)\,\left(
		\young(ji,k)\,T_{mji}g_{kn}
		+ g_{ij}\,(T_{mkn}-2t_mg_{kn})\,.
		\right)
		\\
		\intertext{and}
		\Lambda_{kab} &= \frac16\,\young(ab)\,(T_{kab}-2t_{a}g_{bk})\,.
	\end{align*}
\end{proposition}

Equation~\eqref{eqn:raw.tau} permits us to remove $\bigtau_{ij}$ from the equations. We are therefore left with~\eqref{eqn:prolongation:V} and~\eqref{eqn:prolongation.C}. They form overdetermined systems of partial differential equations for $V$ and $C$, respectively.
Forming a closed system, the systems can be integrated for given initial values of $V$, $V_{k}$ and, respectively, $C_{ij}$, at a point, provided that the integrability conditions are satisfied, i.e.~if the respective Ricci identities hold.
The integrability conditions can be easily inferred using results from~\cite{KSV2024}, particularly Proposition 4.1, Theorem 5.1 and Lemma 5.4 therein.
\begin{proposition}\label{prop:SICC}
	In dimension $n=2$, the Ricci identities for~\eqref{eqn:prolongation.C} and~\eqref{eqn:prolongation:V} are equivalent to the decomposition
	\begin{subequations}
		\label{eq:SICC}
		\begin{equation}\label{eq:SICC:T}
			T_{ijk} = S_{ijk}+{\young(ij)}_\circ t_ig_{jk}\,,		
		\end{equation}
		where $S$ is a totally symmetric and trace-free $(0,3)$-tensor and where the $t_i$ are the components of the differential of a function $t$, together with the following conditions:
		\begin{align}
			\label{eq:SICC:Delta.t}
			\Delta t &= \frac13\,S^{abc}S_{abc}+\frac32R\,,
			\\
			\label{eq:SICC:DS.sym}
			{\young(ijkl)}_\circ\nabla_lS_{ijk}
			&=  \frac43\,{\young(ijkl)}_\circ S_{ijk}t_l
			\\
			\label{eq:SICC:tau}
			\bigtau_{ij} &= \left(
				\frac23\,\nabla^2_{ij}t
				- \frac49\,S_{ija}t^a
				- \frac89\,t_it_j
				+ \frac13\,\nabla^aS_{ija}
			\right)_\circ\,,
			\\
			\label{eq:SICC:q}
			0 &= \young(k,l) \bigl(
				q\indices{_k^n_{,l}}
				+T\indices{_{ml}^n}q\indices{_k^m}
				+2t_kq\indices{_l^n}
				+g_k^n\gamma_l
			\bigr)\,,
			\\
			\label{eq:SICC:gamma}
			0 &= \young(i,j)\Bigl(
			\gamma_{i,j}+q\indices{_i^m}\bigtau_{mj}+t_i\gamma_j
			\Bigr)\,.
		\end{align}
	\end{subequations}
\end{proposition}
\begin{proof}
	In Proposition 4.1 of~\cite{KSV2024} the Ricci identities for~\eqref{eqn:prolongation:V} are shown to be equivalent to the system
	\begin{subequations}\label{eq:SICC:V}
		\begin{align}
			\label{eq:SICC:V:T.gen}
			\young(j,k)\Bigl(T_{ijk}+g_{ij}t_k\Bigr)&=0\\
			\label{eq:SICC:V:Q.gen}
			\young(j,k)\Bigl(Q_{ijkl}+g_{ij}q_{kl}\Bigr)&=0\\
			\label{eq:SICC:V:Gamma.gen}
			\young(j,k)\Bigl(
			\Gamma_{ijk}+g_{ij}\gamma_k
			\Bigr)&=0\\
			\label{eq:SICC:V:q.gen}
			\young(k,l)
			\bigl(
			q\indices{_k^n_{,l}}
			+T\indices{_{ml}^n}q\indices{_k^m}
			+2t_kq\indices{_l^n}
			+g_k^n\gamma_l
			\bigr)&=0\\
			\label{eq:SICC:V:gamma.gen}
			\young(i,j)\Bigl(
			\gamma_{i,j}+q\indices{_i^m}\bigtau_{mj}+t_i\gamma_j
			\Bigr)&=0.
		\end{align}
	\end{subequations}
	Equation~\eqref{eq:SICC:T} follows from~\eqref{eq:SICC:V:T.gen}. Resubstituting~\eqref{eq:SICC:T} into the other conditions, for dimension~2 specifically, we find that~\eqref{eq:SICC:V:Q.gen} and~\eqref{eq:SICC:V:Gamma.gen} are already satisfied.
	Using~\eqref{eqn:raw.tau} and the decomposition of $T$, a direct computation yields the formula~\eqref{eq:SICC:tau}.
	Finally, it was proven in~\cite{KSV2024} that the Ricci identity for~\eqref{eqn:prolongation.C} is equivalent to
	\begin{equation}
		\label{eq:SICC:K}
		\young(k,l)
		\Bigl(
		P\indices{_{ijk}^{mn}_{,l}}
		+P\indices{_{ijk}^{pq}}P\indices{_{pql}^{mn}}
		\Bigr)
		=
		\frac12\young(ij)\young(mn)
		R\indices{^m_{ikl}}g^n_j\,.
	\end{equation}
	A direct computation shows that it is equivalent to~\eqref{eq:SICC:Delta.t} and~\eqref{eq:SICC:DS.sym}, given \eqref{eq:SICC:T}, \eqref{eq:SICC:tau}, \eqref{eq:SICC:q} and~\eqref{eq:SICC:gamma}.
\end{proof}

Given that the integrability conditions are satisfied, the differential equations~\eqref{eqn:prolongation:V} and~\eqref{eqn:prolongation.C} allow one to reconstruct the compatible potentials $V$ and tracefree conformal Killing tensors $C$. For this reason we refer to $T_{ijk}$ and $\bigtau_{ij}$ as the \emph{primary} and \emph{secondary structure tensor}, respectively.
Similarly $S_{ijk}$ will be called the \emph{conformal structure tensor}. This name is motivated further in Section~\ref{sec:conformal.trafos}.

\subsection{Local description}\label{sec:local.conf.equations}
We review the discussion in Section~\ref{sec:2D.systems} on the level of specific local coordinates~\eqref{eqn:local.coordinates}.
In such local coordinates we have simple expressions for the structure tensors, i.e.~for the tensor fields $T$ and $\bigtau$ in~\eqref{eq:conformal.Wilczynski} as well as for the tensor field $S$ and the differential $dt$ in~\eqref{eq:SICC:T}:
\begin{align*}
	S &= \phi^2(s_{1}(z,w)dz^3+s_{2}(z,w)dw^3)
	\\
	dt &= \diff{t}{z}dz+\diff{t}{w}dw = t_1\,dz+t_2\,dw
	\\
	\bigtau &= \smalltau_1(z,w)dz^2+\smalltau_{2}(z,w)dw^2
	\\
	\intertext{We similarly may write down a coordinate representation for any of the tracefree conformal Killing tensors $C$ compatible with the conformally superintegrable system:}
	C &= \phi^4(c_1(z,w)dz^2+c_2(z,w)dw^2)
	= C_{11}(z,w)dz^2+C_{22}(z,w)dw^2
\end{align*}
Note that there are two coordinate expressions given for $S$ and $C$. The motivation for defining $s_1,s_2$ and $c_1,c_2$ will become transparent in the next section: in contrast to the obvious components $S_{111},S_{222}$ and $C_{11},C_{22}$ they are conformal invariants.
To keep track of the correspondences, the reader might find Table~\ref{tab:correspondences} helpful, which provides a synopsis of the tensorial equations vis-\'a-vis their coordinate counterparts.
\begin{table}
	\caption{\bfseries Corresponding expressions in tensorial and coordinate form}
	\label{tab:correspondences}
	\begin{tabular}{l|c|c}
		Equation & Tensorial Version & Coordinate Version
		\\
		\hline\hline
		Wilczynski Equations & \eqref{eqn:prolongation:V} &  \eqref{eqn:conformal.Wilcyzynski.local} and \eqref{eqn:conformal.Wilcyzynski.local.2}
		\\
		\hline
		Prolongation for $C$ & \eqref{eqn:prolongation.C} & \eqref{eqn:killing.shortcut}
		\\
		\hline
		Decomposition of $T$ & \eqref{eq:SICC:T} & \eqref{eqn:conformal.Wilcyzynski.local}
		\\
		\hline
		Component of $\nabla S$ & \eqref{eq:SICC:DS.sym} & \eqref{eq:DS.local.gen}
		\\
		\hline
		Formula for $\bigtau$ & \eqref{eq:SICC:tau} & \eqref{eq:aleph.eqns}
		\\
		\hline
		Formula for $\Delta t$ & \eqref{eq:SICC:Delta.t} & \eqref{eq:Del.t.gen}
	\end{tabular}
\end{table}
Using the outlined notation, the only non-vanishing components of \eqref{eqn:prolongation.C} are
\begin{subequations}\label{eqn:killing.shortcut}
\begin{align}
	\diff{c_1}{z} &= 0
	\label{eq:c1z}
	\\
	\diff{c_1}{w} &= \frac43s_1c_2 -4\,\diff{}{w}\left( \frac13t+\ln|\phi| \right)c_1
	\\
	\diff{c_2}{z} &= \frac43s_2c_1 -4\,\diff{}{z}\left( \frac13t+\ln|\phi| \right)c_2
	\\
	\diff{c_2}{w} &= 0.
	\label{eq:c2w}
\end{align}
\end{subequations}
Writing out~\eqref{eq:SICC:tau}, we obtain
\begin{subequations}\label{eq:aleph.eqns}
\begin{align}
	\label{eq:aleph.eqns.1}
	\smalltau_{1}
	&= -\frac89\left(\diff{t}{z}\right)^2
	-\frac89\diff{t}{w}s_1
	+\frac23\diff{s_1}{w}
	+\frac23\frac{\del^2t}{\del z^2}
	-\frac43\diff{t}{z}\phi^{-1}\diff{\phi}{z}
	+\frac43s_2\phi^{-1}\diff{\phi}{z}
	\\
	\label{eq:aleph.eqns.2}
	\smalltau_{2}
	&= -\frac89\left(\diff{t}{w}\right)^2
	-\frac89\diff{t}{z}s_2
	+\frac23\diff{s_2}{z}
	+\frac23\frac{\del^2t}{\del w^2}
	-\frac43\diff{t}{w}\phi^{-1}\diff{\phi}{w}
	+\frac43s_2\phi^{-1}\diff{\phi}{z}\,.
\end{align}
\end{subequations}
We remark that this expression can alternatively be obtained directly from~\eqref{eq:dCdV}, analogously to the procedure in \cite{KSV2023}.

Next, use the symmetry of second derivatives for~\eqref{eqn:killing.shortcut}, i.e.
\begin{align*}
	\frac{\del}{\del z}\left( \frac43s_1c_2 -4\,\diff{}{w}\left( \frac13t+\ln|\phi| \right)c_1 \right) &= 0
	\\
	\frac{\del}{\del w}\left( \frac43s_2c_1 -4\,\diff{}{z}\left( \frac13t+\ln|\phi| \right)c_2 \right) &= 0.
\end{align*}
Expand the left hand sides of these equations and use~\eqref{eqn:killing.shortcut} to replace derivatives of $c_1$ and $c_2$. The resulting conditions are linear in $c_1$ and $c_2$ and we note that the coefficients of $c_1$ and $c_2$ in these conditions have to vanish independently, by a similar argument as above, due to the fact that a maximally superintegrable system admits two independent tracefree conformal Killing tensors. We therefore obtain the equations
\begin{subequations}\label{eq:DS.local.gen}
	\begin{align}
		\diff{s_1}{z} &= 4\,\diff{}{z}\left( \frac13t+\ln|\phi| \right)s_1
		\\
		\diff{s_2}{w} &= 4\,\diff{}{w}\left( \frac13t+\ln|\phi| \right)s_2
	\end{align}
\end{subequations}
and (two copies of)
\begin{align}\label{eq:Del.t.gen}
	\frac{\del^2t}{\del z\del w} &= \frac43s_1s_2+\frac38\phi^2R\,,
\end{align}
where $R$ is the scalar curvature of $g$.
By integration, the equations~\eqref{eq:DS.local.gen} yield
\begin{equation}\label{eq:beta}
	s_1(z,w)=\beta_1(w)\phi^4\exp\left(\frac43t\right)\,,\quad
	s_2(z,w)=\beta_2(z)\phi^4\exp\left(\frac43t\right)
\end{equation}
We are going to see, in Section~\ref{sec:euclidean}, c.f.~\eqref{eqn:Kress.Schoebel.Wilczynski.ABCD} and~\eqref{eq:AB}, that the functions $\beta_1$ and $\beta_2$ do indeed play a prominent role in the existing solution for flat, 2-dimensional superintegrable systems.

We return to our main discussion noting that the condition~\eqref{eq:conformal.Wilczynski} can be written in local coordinates~\eqref{eqn:local.coordinates} in the form
\begin{equation}\label{eqn:conformal.Wilcyzynski.local}
	\begin{pmatrix} V_{,zz} \\ V_{,ww} \end{pmatrix}
	= 2\,\begin{pmatrix} t_1 & s_1 \\ s_2 & t_2 \end{pmatrix}
	\begin{pmatrix} V_{,z} \\ V_{,w} \end{pmatrix}
	+ \begin{pmatrix} \smalltau_{1} \\ \smalltau_{2} \end{pmatrix}\,V\,,
\end{equation}
where $V_{zz}$ and $V_{ww}$ denote the respective components of the tensor $\nabla^2V$.
Recall that~\eqref{eq:conformal.Wilczynski} does not involve the Laplacian $\Delta V$, and therefore~\eqref{eqn:conformal.Wilcyzynski.local} does not contain $V_{zw}$.
An expression similar to~\eqref{eqn:conformal.Wilcyzynski.local} appears in~\cite{Kress&Schoebel} for the flat case.
Writing out~\eqref{eqn:prolongation:V:2}, we obtain
\begin{equation}\label{eqn:conformal.Wilcyzynski.local.2}
	\frac12\,\begin{pmatrix} V_{,zzw} \\ V_{,wwz} \end{pmatrix}
	= \begin{pmatrix} q_{11} & q_{12} \\ q_{21} & q_{22} \end{pmatrix}
	\begin{pmatrix} V_{,z} \\ V_{,w} \end{pmatrix}
	+ \begin{pmatrix} \gamma_1 \\ \gamma_2 \end{pmatrix}\,V
	+ \begin{pmatrix} t_1 \\ t_2 \end{pmatrix}\,V_{,zw}\,,
\end{equation}
with
\begin{align*}
	q_{11} &= q_{22} = 2s_1s_2+\frac{\del^2t}{\del z\del w} = \frac{10}{3}s_1s_2+\frac38\phi^2R\,,
	\\
	q_{12} &= \diff{s_1}{w}+2s_1\diff{t}{w}+\smalltau_2\,,
	& \gamma_1 &= s_1\smalltau_2+\diff{\smalltau_1}{w}\,,
	\\
	q_{21} &= \diff{s_2}{z}+2s_2\diff{t}{z}+\smalltau_1\,,
	& \gamma_2 &= s_2\smalltau_1+\diff{\smalltau_2}{z}\,,
\end{align*}
where $\smalltau_1,\smalltau_2$ are given by~\eqref{eq:aleph.eqns}.
For simplicity we shall, from now on, restrict to the case of real valued tensors.
Therefore we assume that $w=\bar z$ is the complex conjugate of $z$.
This simplifies~\eqref{eqn:conformal.Wilcyzynski.local}, and indeed, in order that $\bigtau$, $S$ and $dt$ be real-valued, we have $\smalltau_{2}=\overline{\smalltau_{1}}$, $s_2=\overline{s_1}$ and $t_{\bar z}=\overline{t_1}$.
This implies that $t$ is real-valued and that
\begin{align*}
	\bigtau &= \smalltau dz^2+\overline{\smalltau}d\bar z^2\,,
	\\
	S &= \phi^2(sdz^3+\overline{s}d\bar z^3)\,,
\end{align*}
where $\smalltau=\smalltau(z,\bar z)$ and $s=s(z,\bar z)$.
Moreover, due to~\eqref{eq:c1z} and~\eqref{eq:c2w}, respectively, $c_1$ becomes holomorphic and $c_2$ anti-holomorphic.
Finally, we note that the equations obtained earlier in this section present themselves as pairs of complex conjugate conditions if $w=\bar z$ .

\subsection{Conformal transformations}\label{sec:conformal.trafos}
We review conformal transformations of non-degenerate systems, which will play a significant role in the subsequent discussions.
In particular, conformally invariant objects will help us with a concise formulation of the integrability conditions.
In general, a conformal transformation can be characterised by
$$ g\longmapsto\Omega^2g\,,\quad V\longmapsto \Omega^{-2}V\,, $$
where $\Omega$ is a scalar function on $M$.
In~\cite{KSV2024} it was shown that the components of the primary structure tensor transform according to
$$ t\longmapsto t-3\ln|\Omega|=:t'
	\qquad\text{and}\qquad
	S_{ijk}\longmapsto \Omega^2 S_{ijk}. $$
Compare this to the definition of the function $s$ in the beginning of Section~\ref{sec:local.conf.equations}: since $S\indices{_{ij}^a}=g^{ab}S_{ijb}$ are components of a conformally invariant tensor field, the function $s$, and hence also $\bar s$, are conformally invariant functions (if we use the same local coordinates after the transformation). In other words, $s$ and $\bar s$ are complex conjugate functions that are well-defined on the Riemann surface $(S,J)$.
Moreover, using local coordinates~\eqref{eqn:local.coordinates} and choosing $\Omega=\exp(-\frac13t)$, we obtain $t'=0$ after the conformal rescaling.
In this way we are able to use conformal transformations in order to eliminate $t$ at the expense of changing the underlying metric. Analogously, if we choose $\Omega=\phi^{-1}$, then the resulting metric after the conformal transformation is flat. This \emph{gauge freedom} will be discussed in detail later.

\begin{lemma}\label{la:conformal.trafo}
	Under a conformal transformation, a non-degenerate conformally superintegrable system given by a metric $g$ and a potential $V$ satisfying~\eqref{eq:conformal.Wilczynski}, transforms according to
	\begin{equation}\label{eqn:trafo.gV}
		g\longmapsto\Omega^2g\,,\qquad V\longmapsto \Omega^{-2}V
	\end{equation}
	and
	\begin{equation}\label{eqn:trafo.StTau}
		\begin{gathered}
			S_{ijk}\longmapsto\Omega^2S_{ijk}\,,\qquad t\longmapsto t-3\Upsilon\,,\qquad
			\\
			\bigtau_{ij}\longmapsto\bigtau_{ij}+2T_{ija}\Upsilon^a
			-{\young(ij)}_\circ\left(\nabla^2_{ij}\Upsilon+2\Upsilon_i\Upsilon_j\right)\,,
		\end{gathered}
	\end{equation}
	where $\Upsilon=\ln|\Omega|$ and $\Upsilon_a=\Upsilon_{,a}$.
\end{lemma}
\noindent The general proof can be found in~\cite{KSV2024}. We reproduce it here for the special case of dimension $n=2$.
\begin{proof}
	The transformations~\eqref{eqn:trafo.gV} hold by definition. We write $\hat g=\Omega^2g$ and $\hat V=\Omega^{-2}V$.
	Note that the Levi-Civita connections $\nabla$ and $\hat\nabla$ of $g$ and $\hat g$, respectively, satisfy $\hat\nabla = \nabla-2d\Upsilon$.
	Now consider
	\begin{align*}
		\hat\nabla^2_{zz}\hat V &=
		\hat\nabla_z\left(\Omega^{-2}V\right)_{,z}
		= \left( \nabla^2_{zz}V -6V_{,z}\Upsilon_{,z} - 2V\,(\nabla^2_{zz}\Upsilon-2\Upsilon_{,z}^2) + 6V\,\Upsilon_{,z}^2 \right)\Omega^{-2}
		\\
		&= \left( 2t_1V_{,z}+2s_1V_{,w}+\smalltau_1V -6V_{,z}\Upsilon_{,z} - 2V\,(\nabla^2_{zz}\Upsilon-2\Upsilon_{,z}^2) + 6V\,\Upsilon_{,z}^2 \right)\Omega^{-2}
		\\
		&= \left[ (2t_1-6\Upsilon_{,z})V_{,z}
					+2s_1V_{,w}
					+(\smalltau_1- 2\nabla^2_{zz}\Upsilon+2\Upsilon_{,z}^2)V
			\right]\Omega^{-2}
	\end{align*}
	where we have used~\eqref{eqn:conformal.Wilcyzynski.local}. As $\hat V$ satisfies an equation analogous to~\eqref{eqn:conformal.Wilcyzynski.local}, i.e.\ in particular
	\[
		\hat\nabla^2_{zz}\hat V=2\hat t_1\hat V_{,z}+2\hat s_1\hat V_{,w}+\hat\smalltau_1\hat V
		=\left[ 2\hat t_1 V_{,z}+2\hat s_1 V_{,w}+\left( \hat\smalltau_1-2\hat t_1\Upsilon_{,z}-2\hat s_1\Upsilon_{,w}\right)V\right]\,\Omega^{-2},
	\]
	where $\hat t_1$, $\hat s_1$ and $\hat\smalltau_1$ have the obvious meaning. 
	We first compare the coefficients of $V_{,z}$, $V_{,w}$ and $V$ in these expressions, obtaining
	\[
		\hat t_1 = t_1-3\Upsilon_{,z} \qquad\text{and}\qquad \hat s_1=s_1\,.
	\]
	as well as
	\[
		\smalltau_1- 2\nabla^2_{zz}\Upsilon+2\Upsilon_{,z}^2
		= \hat\smalltau_1-2\hat t_1\Upsilon_{,z}-2\hat s_1\Upsilon_{,w}
		= \hat\smalltau_1-2 t_1\Upsilon_{,z}+6\Upsilon_{,z}^2-2 s_1\Upsilon_{,w}\,.
	\]
	We conclude
	\[
		\hat\smalltau_1
		= \smalltau_1- 2\nabla^2_{zz}\Upsilon-4\Upsilon_{,z}^2
		+2 t_1\Upsilon_{,z} +2 s_1\Upsilon_{,w}\,.
	\]
	Together with the analogous computation for $\hat\nabla^2_{ww}\hat V$, we confirm~\eqref{eqn:trafo.StTau}.
\end{proof}
\noindent In view of~\eqref{eq:SICC:tau}, the tensor field
$$ \nabla^aS_{ija} = 2\phi^{-1}\left( 2s_1\diff{\phi}{w}+\phi\diff{s_1}{w} \right)dz^2
					+2\phi^{-1}\left( 2s_2\diff{\phi}{z}+\phi\diff{s_2}{z} \right)dw^2\,. $$
is not invariant under a conformal rescaling of the conformally superintegrable system. However, the following modification of the tensor field is.
\begin{lemma}
	For a two-dimensional non-degenerate conformally superintegrable system, the tensor field
	\begin{equation}\label{eq:Xi}
		\Xi_{ij} = \left( \nabla^a + \frac23 t^a\right) S_{ija}
	\end{equation}
	is conformally invariant.
\end{lemma}
\begin{proof}
	We let
	\begin{equation}\label{eq:def.Z}
		\mathsf Z_{ij} = \nabla^aS_{ija}\,.
	\end{equation}
	Considering a transformation $g\to\Omega^2 g$ with $\Upsilon_i=\Omega^{-1}\Omega_{,i}$, and denoting transformed objects by a tilde, we have
	\begin{align*}
		\tilde{\mathsf Z}_{ij}
		&= \tilde{g}^{ab}\tilde{\nabla}_aS_{ijb}
		= g^{ab} \left( 2\Upsilon_{,a} S_{ijb} + \tilde\nabla_aS_{ijb} \right)
		= \mathsf{Z}_{ij} + 2\Upsilon^a S_{ija}\,.
	\end{align*}
	We infer that
	$$ {\mathsf Z}_{ij}+\alpha t^aS_{ija}
	\longrightarrow \mathsf{Z}_{ij}+2\Upsilon^aS_{ija}+\alpha (t^a-3\Upsilon^a)S_{ija}
	= \mathsf{Z}_{ij}+\alpha t^aS_{ija} +(2-3\alpha)\Upsilon^aS_{ija}, $$
	and the claim then follows for $\alpha=\frac23$.
\end{proof}

\begin{example}
	The simplest example of a superintegrable system on a Riemann surface is the harmonic oscillator class (a proper definition is given below).
	Consider the flat metric $g=dzd\bar z$ and the potential
	$$ V=\alpha_0 z\bar z+\alpha_1(z+\bar z)+i\alpha_2(z-\bar z)+\alpha_3\,,\qquad
	\text{where $\alpha_1,\alpha_2,\alpha_3\in\RR$}. $$
	Then $H=g^{-1}(\mathbf p,\mathbf p)+V$, where $\mathbf p$ denotes the canonical momenta, defines the (real-valued) Hamiltonian of the (non-degenerate) Harmonic Oscillator. Its conformal class consists of all metrics $\phi^2dzd\bar z$ with potentials $\phi^{-2}V$. We find
	$$ S_{ijk}=0\quad\text{and thus}\quad \Xi_{ij}=0\,. $$
\end{example}

For later reference, we conclude this section recalling the well-known conformal transformation of the scalar curvature,
$$ R\longrightarrow
\Omega^{-2}(R-2\Delta\ln|\Omega|)\,. $$

\section{Conformal superintegrability}\label{sec:riemann.surface}

In this section we investigate the integrability conditions~\eqref{eq:SICC:Delta.t}, \eqref{eq:SICC:q} and~\eqref{eq:SICC:gamma}, along with the formulas~\eqref{eq:SICC:T} and~\eqref{eq:SICC:tau}, in the specific case of dimension~2.
Note that conformal transformations offer a gauge freedom for non-degenerate conformally superintegrable systems: for example if we choose a conformal rescaling with $\ln|\Omega|=\frac13t$ in Lemma~\ref{la:conformal.trafo}, then we obtain a non-degenerate conformally superintegrable system with trace-free primary structure tensor. A non-degenerate conformally superintegrable system such that the primary structure tensor $T_{ijk}$ is trace-free is said to be \emph{in standard gauge}. Without loss of generality, one may then assume $t=0$, since the equations only involve derivatives of $t$ and $t$ is determined only up to addition of a constant.
Analogously, a non-degenerate conformally superintegrable system on a flat manifold is said to be \emph{in flat gauge}, and proper systems are said to be in \emph{proper gauge}. The existence of a proper gauge system is guaranteed by a result in \cite{Capel_phdthesis}.
\begin{remark}\label{rmk:tau=0}
	For a conformal system in proper scale, it is shown in Lemma~3.15 of~\cite{KSV2024} that $\bigtau_{ij}=0$, for any dimension~$n\geq2$ and regardless of non-degeneracy.
	For non-degenerate systems, the converse statement was proven in Corollary~5.3 of~\cite{KSV2024}.
\end{remark}

For ease of computation, we are going to carry out most computations in specific gauge choices, particularly in standard gauge (i.e.~$t=0$) and flat gauge (i.e.~$\phi^2=1$). Note that these gauge choices are always possible for a 2-dimensional non-degenerate (conformally) superintegrable system.
These specific equations are then generalised performing a suitable conformal transformation.
From now on we set $w=\bar z$. Note that a system in standard gauge thus has the primary structure tensor of the form
$$ T_\text{std} = e^{2u}(sdz^3+\bar sd\bar z^3)\,, $$
where $u$ is a scalar function such that the underlying metric has the form
$$ g_\text{std}=e^{2u}dzd\bar z. $$
On the other hand, in flat gauge we have
$$ T_\text{flat} =
		sdz^3
		+\diff{t}{z}dz\otimes dz\otimes d\bar z
		+\diff{t}{\bar z}d\bar z\otimes d\bar z\otimes dz
		+\bar sd\bar z^3
		\quad\text{and}\quad
		g_\text{flat}=dzd\bar z\,.
$$
Assuming that the underpinning systems are equivalent by virtue of a conformal transformation, $g_\text{std}=\Omega^2g_\text{flat}$ for a function $\Omega=e^u$, and so we identify $t=\frac13u$.

\subsection{Useful identities}

In dimension 2, some trivial tensor identities hold due to the low number of tensor components and help to simplify many expressions considerably. We list some of these trivial identities for later reference.
The curvature tensor in dimension~2, for instance, is determined by the Gau{\ss} curvature $\kappa$ and we have
\begin{equation}\label{eq:curvature.tensor}
	R_{ijkl} = \frac{R}{2}\,(g_{ik}g_{jl}-g_{il}g_{jk}) = \kappa\,(g_{ik}g_{jl}-g_{il}g_{jk})\,,
\end{equation}
where $R$ is the scalar curvature.
Similarly, we have the identities
\begin{equation}
	\label{eq:S.identity}
	S\indices{_{ij}^a} S_{kla}
	= \frac14\,S^{abc}S_{abc}\,\left(
	g_{i k} g_{j l} + g_{i l} g_{j k} - g_{i j} g_{k l}
	\right)\,,
\end{equation}
for a totally symmetric, trace-free tensor field with components $S_{ijk}$, and
\begin{equation}
	\label{eq:Z.identity}
	Z\indices{_i^a}Z_{aj} = \frac12\,g_{ij}\,Z_{ab}Z^{ab}
\end{equation}
for a symmetric, trace-free tensor field $Z_{ij}$.
For the latter tensor fields, the identities
\begin{align}
	0 &= g_{ij}Z_{kl}-g_{ik}Z_{jl}-g_{jl}Z_{ik}+g_{kl}Z_{ij}\,
	\\
	\young(j,k)\nabla_k\Xi_{ij} &= \frac12\,\young(ji,k)g_{ik}\nabla^a\Xi_{aj}\,,
	\label{eq:trivial.identity.hookZ}
	\\
	S_{ija}Z\indices{^a_k} &= \frac12{\young(ij)}_\circ S_{iab}Z^{ab}g_{jk}
\end{align}
hold as well.

\subsection{Standard gauge}\label{sec:standard.conformal}
Our aim now is to formulate the integrability conditions for the compatible potentials and (trace-free) conformal Killing tensors in a concise manner. We do this by first considering specific gauge choices, beginning with the standard gauge.
By a suitable conformal transformation, we transform a conformally superintegrable system to one with $t=0$.
Note that the standard gauge is therefore characterized by the equalities
$$ T_{ijk} = S_{ijk}\,,\qquad \Xi=\div(S)\,. $$

\begin{proposition}\label{prop:structural.equations.conformal.standard}~ 
	For a system in standard gauge the structure tensors forms a closed system
	\begin{align}
		\nabla_lS_{ijk}
		&= \frac16\,{\young(ijk)} \left(
					\Xi_{ij}g_{kl}-\frac12\,g_{ij}\Xi_{kl}
				\right)
		\label{eq:DS.standard}
		\\
		\nabla_k \Xi_{ij}
		&= \frac13 S_{ijk}S^{abc}S_{abc}+\frac23\,{\young(ij)}_\circ\ g_{ik}S_{jab}\Xi^{ab}
		\label{eq:DZ.standard}
		\\
		\intertext{with}
		R &= -\frac29\,S^{abc}S_{abc}
		\label{eq:Del.t.standard}
		\\
		\intertext{as well as}
		\nabla_kR &= -\frac49S_{kab}\Xi^{ab}\,
		\label{eq:DR.standard}
		\\
		\intertext{and}
		\bigtau_{ij} &= -\frac13\,\Xi_{ij}\,.
		\label{eq:tau.standard}
	\end{align}
	The integrability conditions~\eqref{eq:SICC} are then satisfied, as are the integrability conditions for~\eqref{eq:DS.standard} and~\eqref{eq:DZ.standard}.
\end{proposition}
\begin{remark}
	Equation~\eqref{eq:DR.standard} is not independent, but a consequence of~\eqref{eq:Del.t.standard} and~\eqref{eq:DS.standard}.
	Furthermore, it is straightforward to verify that the Hessian of $R$, which is obtained from~\eqref{eq:DR.standard} after a further differentiation, is indeed symmetric thanks to~\eqref{eq:DS.standard} and \eqref{eq:DZ.standard}.
\end{remark}
\begin{proof}
	Equations~\eqref{eq:Del.t.standard}, \eqref{eq:DS.standard} and~\eqref{eq:tau.standard} follow immediately from the discussion in Section~\ref{sec:local.conf.equations}.
	Using~\eqref{eq:SICC:q} together with~\eqref{eq:tau.standard} and~\eqref{eq:trivial.identity.hookZ}, we obtain
	$$ \nabla_k\Xi_{ij} = \frac16\young(ijk)\nabla_k\Xi_{ij} + \frac{2}{3}\,{\young(ij)}_\circ g_{ik}S_{kab}\Xi^{ab}\,, $$
	and furthermore the Ricci identity for~\eqref{eq:DS.standard} together with~\eqref{eq:curvature.tensor} implies
	$$ \young(ijk)_\circ \nabla_k\Xi_{ij} = 2 S_{ijk}S^{abc}S_{abc}\,. $$
	We thus arrive at~\eqref{eq:DZ.standard}.
	It follows by a direct computation that~\eqref{eq:SICC:V:gamma.gen} is already satisfied.
	Also by a direct computation, one confirms that the integrability condition for~\eqref{eq:DZ.standard} is satisfied.
	Finally, \eqref{eq:DR.standard} is obtained from~\eqref{eq:Del.t.standard}, by taking one covariant derivative and then resubstituting~\eqref{eq:DS.standard}.
\end{proof}

We now discuss Proposition~\ref{prop:structural.equations.conformal.standard}. We observe that, for a given metric $g=e^{2u}dzd\bar z$ with $u=u(z,\bar z)$, the equations~\eqref{eq:DS.standard} and~\eqref{eq:DZ.standard} form an overdetermined system of partial differential equations of finite type, which is subject to the additional curvature constraint~\eqref{eq:Del.t.standard}. Of course, by a straightforward computation of the curvature of $g$, we have
$$ R = -8e^{-2u}\frac{\del^2u}{\del z\del\bar z}\,. $$

\begin{lemma}
	If $R$ is constant in standard scale, the system is of the harmonic oscillator class.
\end{lemma}
\begin{proof}
	We prove the claim for $R=0$ first. Then it follows from~\eqref{eq:Del.t.standard} that $S=0$, and thus $\Xi=0$ because of~\eqref{eq:DZ.standard}. Moreover, from~\eqref{eq:tau.standard} we conclude $\bigtau=0$. Integrating the resulting equations, in concrete local coordinates, we find the harmonic oscillator system.
	
	Now, if $\nabla R=0$, it follows that $S_{kab}\Xi^{ab}=0$. In local coordinates~\eqref{eqn:local.coordinates}, the components of this condition are $s\bar\xi=0$ and $\bar s\xi=0$. Hence, either $s=0$ or $\xi=0$ and therefore either	$S=0$ or $\Xi=0$. If $\Xi=0$, then it follows from~\eqref{eq:DZ.standard} that $S=0$. It follows that $R=0$ and so the claim follows as above.
\end{proof}

\subsection{Flat gauge}\label{sec:flat.conformal}
We now reconsider the integrability conditions for the compatible potentials and (trace-free) conformal Killing tensors of a non-degenerate system assuming that the system is in flat gauge. Choosing a flat gauge is always possible, for non-degenerate systems in dimension~2, due to the existence of Gau{\ss}' isothermal coordinates.
Applying such a suitable conformal transformation, we therefore may assume the underlying metric of a system to be flat, i.e.~$R=0$.  Recalling the discussion in Section~\ref{sec:local.conf.equations}, we have $\phi=1$ in flat gauge and 
\begin{align*}
	T_\text{flat} &= s\,dz^3 + t_1\,dz^2\otimes dw + t_2\,d\bar z^2\otimes dz + \bar sd\bar z^3
	\\
	S_\text{flat} &= s\,dz^3 + \bar sd\bar z^3
	\\
	\Xi &= \xi dz^2+\bar\xi d\bar z^2\,.
\end{align*}
Using the results from Sections~\ref{sec:2D.systems} and~\ref{sec:local.conf.equations}, and reconsidering~\eqref{eq:SICC:q} and~\eqref{eq:SICC:gamma}, we arrive at the following theorem.
\begin{proposition}\label{prop:structural.equations.conformal.flat}~ 
	For a system in flat gauge the structure tensors satisfy the prolongation equations
	\begin{align}
		\diff{s}{z} &= \frac43\diff{t}{z}s
		\label{eq:Sz.flat} 
		\\
		\diff{s}{\bar z} &= \frac12\xi-\frac23\diff{t}{\bar z}s
		\label{eq:Sw.flat}
		\\
		\diff{\xi}{z} &= \frac{16}{3}s^2\bar s+\frac43\xi\diff{t}{z}
		\label{eq:Xiz.flat} 
		\\
		\diff{\xi}{\bar z} &= \frac83\,s\bar\xi
		\label{eq:Xiw.flat}
		\\
		\frac{\del^2t}{\del z^2} &= \frac43\left(\diff{t}{z}\right)^2
									+2\diff{t}{\bar z}s
									+\frac12\left(3\smalltau-\xi\right)
		\label{eq:hess.t.flat}
		\\
		\frac{\del^2t}{\del z\del\bar z} &= \frac43s\bar s
		\label{eq:Del.t.flat}
		\\
		\intertext{and}
		\diff{\smalltau}{\bar z}
		&= 2\left(
		-\frac{20}{9}s\bar s\diff{t}{z}
		-\frac49s\left(\diff{t}{\bar z}\right)^2
		-\frac13\xi\diff{t}{\bar z}
		+s(\bar\xi-\overline\smalltau)
		\right)
		\label{eq:aleph.w.flat}
	\end{align}
	and their complex conjugates, where $\smalltau$ is an undetermined (complex) function.
\end{proposition}
\begin{proof}
	The proof is similar to the corresponding statement for standard gauge.
	Equations~\eqref{eq:Sz.flat}, \eqref{eq:Sw.flat}, \eqref{eq:hess.t.flat} and~\eqref{eq:Del.t.flat} are obtained from the discussion in Section~\ref{sec:local.conf.equations}.
	Equation~\eqref{eq:Xiz.flat} is obtained from Equations~\eqref{eq:Sz.flat}, \eqref{eq:Sw.flat} and~\eqref{eq:Del.t.flat} together with the symmetry of second derivatives.
	From~\eqref{eq:SICC:q}, we obtain~\eqref{eq:Xiw.flat}.
	Differentiating Equations~\eqref{eq:hess.t.flat} and~\eqref{eq:Del.t.flat}, and using the symmetry of second derivatives, we then find~\eqref{eq:aleph.w.flat}.
	The integrability conditions for~\eqref{eq:Xiz.flat} and~\eqref{eq:Xiw.flat}, as well as~\eqref{eq:SICC:q} are then already satisfied.
\end{proof}

\noindent We observe that the prolongation in Proposition~\ref{prop:structural.equations.conformal.flat} is not of finite type and $\diff{\smalltau}{z}$ cannot be expressed in terms of lower derivatives. However, we will see that in the case of properly superintegrable systems the prolongation is indeed of finite type, i.e.~we obtain a closed system. This will be presented thoroughly in Proposition~\ref{prop:SIC:K.NLP}, but is already apparent by setting $\smalltau=0$ in Proposition~\ref{prop:structural.equations.conformal.flat}.
\begin{proposition}\label{prop:holomorphic.A}
	Let $D=e^{-\frac43t}$. Then $\beta:=Ds$ is anti-holomorphic,
	$$ \beta=\beta(\bar z). $$
\end{proposition}
\noindent In other words, $\bar\beta$ is a holomorphic function.
\begin{proof}
	From~\eqref{eq:Sz.flat} and~\eqref{eq:hess.t.flat}, we infer $\diff{\beta}{z}=0$.
\end{proof}
Note that from~\eqref{eq:Sw.flat} together with~\eqref{eq:hess.t.flat}, we infer
\begin{align*}
	\diff{\beta}{\bar z}
	&= \frac34\frac{\del^2D}{\del z^2}	-\frac32D\smalltau\,.
\end{align*}
This generalises an observation in~\cite{Kalnins&Kress&Miller} and~\cite{Kress&Schoebel} for proper systems, i.e.~for the special case $\smalltau=0$ in our notation. In this case
$$ \diff{\beta}{\bar z} = \frac34\frac{\del^2 D}{\del z^2} $$
and therefore
$$ \frac{\del^3D}{\del z^3} = 0\quad\text{and}\quad \frac{\del^3D}{\del\bar z^3} = 0\,, $$
since $D$ is real.
This result is crucial in~\cite{Kress&Schoebel} for the characterisation of $D$ as a polynomial on the base manifold.

\subsection{Formulas for arbitrary conformal scale}\label{sec:general.conformal}

Having computed the non-linear prolongation and the integrability conditions in the standard and flat gauge, it remains for us to obtain the respective equations for arbitrary gauge.
These are determined in the current section.
Instead of following the same strategy as in the specific gauge choices discussed earlier, we will exploit that the general formulas have to be conformally equivariant conditions.
We begin with the general formula for~$\bigtau$.

\begin{lemma}\label{la:hess.t}
	For a 2-dimensional non-degenerate second order conformally superintegrable system, we have
	\begin{equation}\label{eq:tau.gen.transformed}
		\bigtau_{ij} =
		{\young(ij)}_\circ\left(\frac13\nabla_jt_i-\frac49t_it_j\right)
		- \frac23 S_{ija}t^a
		- \frac13\,\Xi_{ij}
	\end{equation}
	where $\Xi$ is defined as in~\eqref{eq:Xi}.
\end{lemma}
\begin{proof}
	According to the transformation formulas in Section~\ref{sec:conformal.trafos}, and with $\ln\Omega=-\tfrac13t$, we obtain
	\[
		\bigtau_{ij}+\frac13\Xi_{ij}\longmapsto
		\bigtau_{ij} + \frac23 S_{ijk}t^k-{\young(ij)}_\circ\left(\frac13\nabla_jt_i-\frac49t_it_j\right) 
		+\frac13\Xi_{ij}
	\]
	The claim then immediately follows from~\eqref{eq:tau.standard}, since we obtain
	\[
		0 = \bigtau_{ij} + \frac23 S_{ija}t^a -{\young(ij)}_\circ\left(\frac13\nabla_jt_i-\frac49t_it_j\right) +\frac13\,\Xi_{ij}\,,
	\]
	which is equivalent to the asserted formula.
\end{proof}

Next, we obtain the general form of~\eqref{eq:Del.t.standard} and~\eqref{eq:Del.t.flat}.
\begin{remark}
	For a 2-dimensional non-degenerate second order conformally superintegrable system, we have already found that
	\begin{align}\label{eq:Del.t.general}
		\Delta t-\frac32R &= \frac13\,S^{abc}S_{abc},
	\end{align}
	i.e.~\eqref{eq:SICC:Delta.t}.
	The same formula is found if we transform \eqref{eq:Del.t.standard} using a conformal transformation with $t=-\frac13\ln|\Omega|$. Indeed,
	$$ R+\frac29\,S^{abc}S_{abc}\longrightarrow\Omega^{-2}(R-\frac23\Delta t+\frac29S^{abc}S_{abc}). $$
\end{remark}
Lemma~\ref{la:hess.t} and this remark together express second covariant derivatives of $t$ in terms of $S$, $\Xi$, $\nabla t$, $\bigtau$ and $R$.

\begin{proposition}\label{prop:structural.equations.conformal}
	For a 2-dimensional non-degenerate second order conformally superintegrable system, we have the prolongation
	\begin{align}
		\nabla_lS_{ijk}
		&=  \frac16\,{\young(ijk)}_\circ\left[
				-\frac23S_{ijk}t_l+2t_iS_{jkl}
				+\Xi_{ij}g_{kl}
			\right]
		\label{eq:DS.general}
		\\
		\nabla_k\Xi_{ij}
		&= \frac13S_{ijk}\,\,S^{abc}S_{abc}
				-\frac29\young(ijk)\left(
					\Xi_{jk}t_i-\frac12g_{ij}\Xi_{ka}t^a
				\right)
				+\frac23{\young(ij)}_\circ g_{jk}S_{iab}\Xi^{ab}
			\label{eq:DXi.general}
		\\
		\nabla^a\bigtau_{ak}
		&= -S_{kab}\bigtau^{ab}
		+S_{kab}\Xi^{ab} -\frac23\,\Xi_{ka}t^a - \frac49\,S_{kab}t^at^b - \frac59\,S_{abc}S^{abc}t_k
		+\frac12\nabla_kR- Rt_k
		\label{eq:divTau.general}
		\\
		\nabla^2_{ij}t
		&= \frac32\left[
				\bigtau_{ij}+\frac23S_{ija}t^a+\frac49{\young(ij)}_\circ t_it_j+\frac13\Xi_{ij}
			\right]
			+\frac12g_{ij}\left[ \frac13S^{abc}S_{abc}+\frac32R \right]
			\label{eq:DDt.general}
	\end{align}
	where $\Xi$ is the conformal invariant~\eqref{eq:Xi}.
\end{proposition}
\begin{proof}
	We have, with $\Upsilon_i=-\frac13t_i$,
	\begin{align*}
		\nabla_lS_{ijk}
		&\to \Omega^2\nabla_lS_{ijk}+\Omega^2\young(ijkl)S_{ijk}\Upsilon_l
		\\
		&\quad
		-\Omega^2\young(ijk)(S_{ijk}\Upsilon_l+S_{ija}\Upsilon^ag_{kl})
	\end{align*}
	and
	$$ {\young(ijk)} \left(
		\Xi_{ij}g_{kl}-\frac12\,g_{ij}\Xi_{kl}
	\right)
	\to \Omega^2{\young(ijk)} \left(
		\Xi_{ij}g_{kl}-\frac12\,g_{ij}\Xi_{kl}
	\right)\,, $$
	from which we obtain
	\begin{align*}
		\nabla_lS_{ijk}
		&=  \frac16\,\young(ijk)\left[
			-\frac23S_{ijk}t_l+2\left(t_iS_{jkl}-\frac12g_{ij}S_{kla}t^a\right)
			+\left( \Xi_{ij}g_{kl}-\frac12\,g_{ij}\Xi_{kl} \right)
			\right].
	\end{align*}
	Equation~\eqref{eq:DS.general} follows immediately.
	A similar computation yields~\eqref{eq:DXi.general} and~\eqref{eq:divTau.general}.
	Alternatively, \eqref{eq:DXi.general} is obtained from a direct computation, analogous to the corresponding computations in the proofs of Propositions~\ref{prop:structural.equations.conformal.standard} and~\ref{prop:structural.equations.conformal.flat}, subsequently substituting 
	$$\div(\bigtau)=2\phi^{-2}\left(\diff{\smalltau}{\bar z}dz+\diff{\overline\smalltau}{z}d\bar z\right).$$
	A similar direct computation yields~\eqref{eq:divTau.general}.
	Equation~\eqref{eq:DDt.general} is a combination of~\eqref{eq:tau.gen.transformed} and~\eqref{eq:Del.t.general}.
\end{proof}

\noindent Here we observe another fundamental structural difference of 2-dimensional non-degenerate systems, compared to abundant systems in dimension $n\geq3$: while in higher dimensions all first derivatives of the structure tensor $T$ are expressed in terms of $T$ and the curvature, this is not possible in dimension two.
Specifically, the divergence $\mathsf Z_{ij}=\nabla^aS_{ija}$ remains unspecified. However, for the derivative of this divergence, there are, again, expressions involving only $T$, the curvature and $\mathsf Z$, yielding a closed system again.
The above computations have confirmed the following:

\begin{lemma}\label{la:conformally.invariant.structure.functions}
	Consider a second order superintegrable system on a Riemann surface.
	Let $g=e^{2f}dzd\bar z$ be a local realisation of the corresponding conformal class $[g]$, and compute the functions
	$s$, $\xi$ and $t$ as above.
	Let $g'=e^{2f'}dzd\bar z$ be another local realisation of $[g]$, and compute $s'$, $\xi'$ and $t'$ analogously.
	Then
	$$ s'=s\,,\qquad \xi'=\xi \qquad\text{and}\qquad \frac13t'+f'=\frac13t+f+c $$
	where $c\in\RR$.
\end{lemma}
\noindent Note that the constant $c\in\RR$ is an irrelevant as $t$ is only defined up to a constant.
\begin{proof}
	The conformal invariance of $s,\bar s$ and $\xi,\bar\xi$ is granted by construction.
	The standard scale system corresponding to the given realisation is obtained by a rescaling $\Omega$ with $\Upsilon=\ln|\Omega|$ and
	$$ t=3\Upsilon, $$
	which is defined up to addition of a constant.
	Thus with $\Omega=e^{\frac13t}$ we have
	$$ g\to e^{\frac23t}e^{2f}dzd\bar z $$
	and conclude that $u:=f+\frac13t$ characterises the standard scale. Since the standard scale is unique up to adding a constant to $u$, this is unchanged by a conformal rescaling.
\end{proof}

\noindent Because of Lemma~\ref{la:conformally.invariant.structure.functions} we interpret $s$, $\xi$ and $u:=f+\frac13t$ as conformally invariant \emph{structure functions} of the system.

\section{Proper superintegability}\label{sec:proper}

We now consider a special case of the previous discussion, namely when 
$\bigtau_{ij}=0$. This is the case of properly superintegrable systems.

\subsection{Preliminaries}\label{sec:proper.preliminaries}

We recall the definition and properties of proper second order superintegrable systems. A result from~\cite{KSV2023} then allows us to obtain our results for proper systems directly from those in Section~\ref{sec:conformal.systems} by setting $\tau=0$.

\begin{definition}
	\begin{enumerate}
		\item
		A \emph{maximally superintegrable system} in dimension two is a Hamiltonian system admitting $2n-1=3$ functionally independent constants of motion
		$F^{(\alpha)}$,
		\begin{align}
			\label{eq:integral}
			\{F^{(\alpha)},H\}&=0&
			\alpha&=0,1,2,
		\end{align}
		one of which is the Hamiltonian itself. By convention, $F^{(0)}=H$.
		\item
		A constant of motion is \emph{second order} if it is of the form
		\begin{equation}
			\label{eq:quadratic}
			F^{(\alpha)}=K^{(\alpha)}+V^{(\alpha)},
		\end{equation}
		where
		\[
		K^{(\alpha)}(\mathbf p,\mathbf 
		q)=\sum_{i=1}^nK^{(\alpha)}_{ij}(\mathbf q)p^ip^j
		\]
		is quadratic in momenta and $V^{(\alpha)}=V^{(\alpha)}(\mathbf q)$
		a function depending only on positions.
		\item
		A superintegrable system is \emph{second order} if its constants
		of motion $F^{(\alpha)}$ are second order and 
		if~\eqref{eq:Hamiltonian} is given by the Riemannian metric 
		$g_{ij}(\mathbf q)$ on the underlying manifold.
		\item
		We call $V$ a \emph{superintegrable potential} if the Hamiltonian
		\eqref{eq:Hamiltonian} defines a superintegrable system.
	\end{enumerate}
\end{definition}

\noindent Analogous to the case of conformal superintegrability, we obtain from the condition~\eqref{eq:integral} for a second-order constant of the motion $F^{(\alpha)}$:

The component of~\eqref{eq:integral} cubic in momenta is $\{K^{(\alpha)},G\}=0$ and implies that $K_{ij}$ are components of a Killing tensor.
\begin{definition}
		A (second order) \emph{Killing tensor} is a symmetric tensor field on a
		Riemannian manifold satisfying the Killing equation
		\begin{equation}
			\label{eq:Killing}
			K_{ij,k}+K_{jk,i}+K_{ki,j}=0.
		\end{equation}
\end{definition}

The component of~\eqref{eq:integral} linear in momenta is $\{K^{(\alpha)},V\}+\{V^{(\alpha)},G\}=0$ and implies
\[
		dV^{(\alpha)}=K^{(\alpha)}dV.
\]
Differentiation once more, one finds
\begin{equation}\label{eq:dKdV}
		d(K^{(\alpha)}dV)=0
\end{equation}
analogous to~\eqref{eq:dCdV}.
Here we have again identified symmetric forms and endomorphisms using the metric $g$.
In index notation, we have
\begin{equation}
		\label{eq:dKdV:ij}
		\young(i,j)
		\bigl(K\indices{^m_i}V_{,jm}+K\indices{^m_{i,j}}V_{,m}\bigr)=0.
\end{equation}
In local coordinates \eqref{eqn:local.coordinates}, we obtain the condition
\begin{equation}
		\label{eq:dKdV.local}
		K_{11} V_{ww} - K_{22} V_{zz}
		+ \left( \frac32\,K_{11w} - 2\phi^{-1}\phi_wK_{11} \right) V_w
		- \left( \frac32\,K_{22z} - 2\phi^{-1}\phi_zK_{22} \right) V_z = 0
\end{equation}
after substituting the conditions for $K_{ij}$ being a Killing tensor.

\subsection{Non-degenerate systems}\label{sec:proper.non-degeneracy}
Analogous to the conformal case, a second-order conformally superintegrable system is said to be \emph{irreducible} if the space $\mathcal{K}$ of Killing tensors associated to it via the integrals of motion forms an irreducible set of endomorphisms.
Now consider~\eqref{eq:dKdV}. For an irreducible system, one can solve~\eqref{eq:dKdV} for all second derivatives of the potential $V$, except for the Laplacian $\Delta V$, and the solution can be written in terms of the Killing tensors and their derivatives. One therefore obtains an expression of the form
\begin{equation}\label{eq:Wilczynski}
	V_{,ij} = T\indices{_{ij}^k}V_{,k} + \frac1n\,g_{ij}\Delta V\,,
\end{equation}
where the tensor $T_{ijk}$ only depends on the space spanned by the Killing tensors $K^{(\alpha)}$.

Analogously to Definition~\ref{def:non-degenerate.conformal}, we define the non-degeneracy of 2-dimensional irreducible second-order \emph{properly} superintegrable systems as follows.

\begin{definition}\label{def:non-degenerate.proper}
	Consider the quadruple $(M,g,\mathcal V,\mathcal F)$ composed of a smooth 2-dimensional manifold $M$ with metric $g$, a 3-dimensional space of integrals $\mathcal F$ and a 4-dimensional space of potentials $\mathcal V$. We say that it is \emph{non-degenerate} if, for any $V\in\mathcal V$, and any $F\in\mathcal F$, Equations~\eqref{eq:dKdV} and~\eqref{eq:Wilczynski} are satisfied.
	
	An irreducible second-order properly superintegrable system is called \emph{non-degenerate} if such a non-degenerate quadruple exists where $\mathcal V$ and $\mathcal K$ include, respectively, its potential and associated Killing tensors.
\end{definition}

Because of the symmetries of $V_{,ij}$, and again analogously to the case of conformally superintegrable systems, we infer that the structure tensor $T$ in~\eqref{eq:Wilczynski} of a non-degenerate proper system is unique~\cite{KSV2023}, and that it satisfies
\[
	\young(i,j)\,T\indices{_{ij}^k}=0\,,
	\qquad
	T\indices{_a^{ak}} = 0\,.
\]

Resubstituting~\eqref{eq:Wilczynski} into~\eqref{eq:dKdV}, one finds
\begin{equation}
	\label{eq:prolongation:K}
	K_{ij,k}=\frac13\young(ji,k)T\indices{^m_{ji}}K_{mk},
\end{equation}
again in a manner analogous to the conformal case, see also~\cite{KSV2023}.

In the following we profit from a result shown in \cite{KSV2024} for \emph{non-degenerate} systems, mentioned earlier in Remark~\ref{rmk:tau=0}.
It allows us to consider the proper system as a conformal system and replace the Killing tensors (except the metric) by trace-free conformal Killing tensors. The integrability conditions for~\eqref{eq:Wilczynski} and~\eqref{eq:prolongation:K} then follow immediately from those of~\eqref{eq:conformal.Wilczynski} and~\eqref{eqn:prolongation.C}, setting $\tau=0$. Moreover, \eqref{eqn:raw.tau} provides us with an additional constraint on $T$.

\subsection{Integrability conditions}\label{sec:proper.integrability}

Due to Remark~\ref{rmk:tau=0}, we are able to directly formulate the integrability conditions for $V$ and $K_{ij}$ analogously to the case of conformal superintegrability. We find the following non-linear prolongation for $T_{ijk}$.

\begin{proposition}\label{prop:SIC:K.NLP}
For a non-degenerate second-order properly superintegrable system in two dimensions, the structure tensor satisfies the prolongation (``closed system'')
\begin{subequations}\label{eqn:NLP.abstract}
	\begin{align}
		\Delta t &= \frac32\,R + \frac13\,S^{abc}S_{abc}
		\label{eqn:Delta.t.2D}
		\\
		\mathring\nabla^2_{ij}t &=
		\frac13S_{ija}t^a
		+ \frac43\,\left( t_it_j \right)_\circ
		- \frac12\,\Xi_{ij}
		\label{eqn:tij.2D}
		\\
		\nabla_lS_{ijk}
		&= \frac16\,{\young(ijk)}_\circ \left(
			-\frac23\,S_{ijk}t_l
			+2\,t_iS_{jkl}
			+ \Xi_{ij}g_{kl} 
		\right)
		\label{eqn:DS.2D}
		\\
		\nabla_k\Xi_{ij}
		&= \frac13\,S_{ijk}S^{abc}S_{abc}
		-\frac29\,\young(ijk)\,(\Xi_{jk}t_i-\frac12\,\Xi_{ka}t^a)
		+\frac23\,{\young(ij)}_\circ\,g_{jk}S_{iab}\Xi^{ab}
		\label{eqn:DXi.2D}
	\end{align}
\end{subequations}
and the algebraic condition
\begin{equation}\label{eqn:Remn}
	\frac23\,\Xi_{ia}t^a + \frac49\,S_{iab}t^at^b + \frac59\,S_{abc}S^{abc}t_i - S_{iab}\Xi^{ab}
	= \frac12\,\nabla_iR - Rt_i\,.
\end{equation}
\end{proposition}
\noindent Two aspects of the prolongation equations in Proposition~\ref{prop:SIC:K.NLP} deserve to be mentioned in contrast to their conformal counterpart in Proposition~\ref{prop:structural.equations.conformal}: first, while the equations in the conformal case form a closed system. Second, the initial data of the conformal system is not subject to algebraic conditions, but the proper system is subjected to~\eqref{eqn:Remn}.
\begin{proof}
	This follows immediately from Propositions~\ref{prop:structural.equations.conformal.standard}, \ref{prop:structural.equations.conformal.flat} and~\ref{prop:structural.equations.conformal} after setting $\bigtau=0$.
	In particular, \eqref{eqn:Remn} is obtained from~\eqref{eq:divTau.general}.
\end{proof}

Let $g=\phi^2dzd\bar z$ denote the metric in local coordinates~\eqref{eqn:local.coordinates}.
Moreover, let
\begin{align*}
	S &= s dz^3+\bar sd\bar z^3\,,
	&
	\Xi &= \xi dz^2+\bar\xi d\bar z^2\,.
\end{align*}
Then Equation~\eqref{eqn:Remn} becomes
\begin{equation}\label{eqn:Remn.complex}
 		\frac{80}{9}s\bar s t_1
 		+\frac{16}{9}st_{\bar z}^2
	 	-4 s\bar\xi
 		+\frac43\xi t_{\bar z}
 		= \frac12\phi^2R_z - R\phi^2 t_1\,.
\end{equation}
We see here that this equation cannot be expressed simply using the complex functions specifying the structure tensor, i.e.~the functions $s$ and $\xi$ as well as $t_1$ and $t_2$. Indeed, \eqref{eqn:Remn.complex} involves the (real-valued) functions $\phi$ and $R$ as well as the (complex-valued) derivatives of $R$.

\begin{lemma}
	(i) The differential consequences of~\eqref{eqn:Remn} are given by
	\begin{equation}\label{eqn:DRemn}
		\frac{27}{2}\,\lVert\Xi\rVert^2+\lVert S\rVert^4-32\,S_{abk}\Xi^{ab}t^k 
		+ \frac43\,\Xi_{ab}t^at^b + \frac{16}9\,S_{abc}t^at^bt^c + \frac{37}9\,\lVert S\rVert^2\lVert t\rVert^2
		+ 2R\lVert t\rVert^2 = 0\,
	\end{equation}
	and its differential consequences,
	where
	\begin{align*}
		\lVert S\rVert^2&=S^{abc}S_{abc}\,,
		&
		\lVert t\rVert^2&=t^at_a\,,
		&
		\lVert\Xi\rVert^2&=\Xi^{ab}\Xi_{ab}\,.
	\end{align*}

	\noindent (ii) If $R=0$, then the integrability conditions are precisely
	\begin{align*}
		\frac23\,\Xi_{ia}t^a + \frac49\,S_{iab}t^at^b + \frac59\,\lVert S\rVert^2 t_i - S_{iab}\Xi^{ab}
		&= 0
		\\
		\frac{27}{2}\,\lVert\Xi\rVert^2+\lVert S\rVert^4-20\,\Xi_{ab}t^at^b -\frac{112}{9}\,S_{abc}t^at^bt^c -\frac{123}{9}\,\lVert S\rVert^2\lVert t\rVert^2
		&= 0
	\end{align*}
\end{lemma}
\begin{proof}
	Consider~\eqref{eqn:Remn}. We take the differential and obtain, due to the symmetry of the Hessian of $R$,
	\begin{equation}\label{eq:ddR=0}
		\frac{27}{2}\,\Xi_{ab}\Xi^{ab}+S^4-30\,S_{abk}\Xi^{ab}t^k + 3\,S^2t^at_a + t^a\nabla_aR = 0\,.
	\end{equation}
	which is simply an algebraic condition in $\Xi$, $S$ and $\nabla t$.
	The first claim then follows since, after having taken another derivative, we may replace any derivatives of these objects using~\eqref{eqn:NLP.abstract} and~\eqref{eqn:Remn}.
	
	For the second claim, take a further covariant derivative of~\eqref{eqn:DRemn} and obtain
	\begin{equation}\label{eqn:DDRemn}
		2\lVert t\rVert^2\,\nabla_lR + 2t_l\,R^2+\left( 3\lVert S\rVert^2\,t_l-14\,S_{lab}\Xi^{ab} + \frac{16}{9}S_{lab}t^at^b \right)R = 0\,.
	\end{equation}
	It follows that, in the case of vanishing curvature, the only algebraic conditions are~\eqref{eqn:Remn} and~\eqref{eqn:DRemn}, which are algebraic in the structure tensors $S$, $\Xi$ and $dt$. Note that~\eqref{eqn:DRemn} then simplifies using~\eqref{eqn:Remn}.
\end{proof}

The lemma states that the integrability conditions for a proper system are given by~\eqref{eqn:Remn} as well as~\eqref{eqn:DRemn} and all the differential consequences of the latter, which is a \emph{single} equation.
Moreover, if $R=0$, only three equations need to be considered. A solution in a point $x_0\in M$ can thus locally be extended to a solution of~\eqref{eqn:NLP.abstract}. We will demonstrate in the next section that this leads to the known structural equations for flat surfaces.

Next reconsider~\eqref{eqn:DDRemn} assuming $R$ is constant with $R\ne0$. We may then take either $R=+1$ or $R=-1$, making~\eqref{eqn:Remn} an algebraic equation in the structure tensors alone. This case will be discussed further in Section~\ref{sec:sphere}, while a full solution is left for future work.

The case of non-constant $R$ is more involved, but in this case one may interpret~\eqref{eqn:Remn} as an equation for $\nabla R$,
\begin{equation}\label{eqn:dR}
	\nabla_iR
	= \frac43\,\Xi_{ia}t^a + \frac89\,S_{iab}t^at^b + \frac{10}{9}\,S_{abc}S^{abc}t_i - 2S_{iab}\Xi^{ab}
	+ 2Rt_i\,.
\end{equation}
This makes~\eqref{eqn:DRemn} and~\eqref{eqn:DDRemn} algebraic in the structure tensors $S$, $\Xi$ and $dt$ \emph{and} the curvature $R$.
Taking a further covariant derivative, using the Ricci identity, and employing \eqref{eqn:NLP.abstract} and~\eqref{eqn:dR} to substitute derivatives of $S$, $\Xi$, $dt$ and $R$, we are then left with solving purely algebraic equations for $S$, $\Xi$, $dt$ and $R$. A further study of these equations is left for future research.

\section{Two-dimensional Euclidean space}\label{sec:euclidean}

Let us now restrict to flat space, $R=0$, and compare the above equations with the expressions in the literature \cite{Kress&Schoebel,KKM07b}.
We take the metric in the form $g=dzdw$, and have thus
\begin{equation}
\label{eqn:Kress.Schoebel.Wilczynski}
 \begin{pmatrix}
  V_{zz} \\ V_{ww}
 \end{pmatrix}
 =
 \begin{pmatrix}
  T_{112} & T_{111} \\
  T_{222} & T_{221}
 \end{pmatrix}
 \begin{pmatrix}
  V_z \\ V_w
 \end{pmatrix}
 =
 2\,\begin{pmatrix}
  t_1 & s_1 \\
  s_2 & t_2
 \end{pmatrix}
 \begin{pmatrix}
  V_z \\ V_w
 \end{pmatrix}
\end{equation}
On the other hand, \cite{Kress&Schoebel} write
\begin{equation}
 \label{eqn:Kress.Schoebel.Wilczynski.ABCD}
  \frac23\,
  \begin{pmatrix}
   V_{zz} \\ V_{ww}
  \end{pmatrix}
  =
  \begin{pmatrix}
   C_{11} & C_{12} \\ C_{21} & C_{22}
  \end{pmatrix}
  =
  \frac1D\,
  \begin{pmatrix}
   -D_z & A_z \\ B_w & -D_w
  \end{pmatrix}
\end{equation}
Note that these equations do not imply any constraints on $\Xi$, nor on $\nabla^aS_{ija}$.
Thus, in coordinates,
\[
 \partial_zS_{222}\,,\quad \partial_wS_{111}
\]
remain unspecified for the time being.
However, we also have~\eqref{eqn:NLP.abstract} and \eqref{eqn:Remn}.
Equation~\eqref{eqn:tij.2D} together with~\eqref{eqn:DS.2D} provides us with the equations (8)---(13) in~\cite{KKM07b}.
\bigskip

\noindent Let us discuss Equations~\eqref{eqn:Delta.t.2D}--\eqref{eqn:DS.2D} in more detail.
Firstly, Equation~\eqref{eqn:Delta.t.2D} corresponds to the Pl\"ucker relation (3.11a) in~\cite{Kress&Schoebel}.
Indeed, define
\begin{equation}
	\mathsf{d} = \exp\left(-\frac43t\right)
\end{equation}
and
\begin{equation}
	\mathsf{s}_1=\frac{2}{3}\,S_{111}\mathsf{d}
	\qquad\text{and}\qquad
	\mathsf{s}_2=\frac{2}{3}\,S_{222}\mathsf{d}\,,
\end{equation}
as an ansatz. Note that $\mathsf{s}_1$ and $\mathsf{s}_2$ are almost identical to $\beta_1,\beta_2$ in~\eqref{eq:beta}, see also Proposition~\ref{prop:holomorphic.A}.
Using the ansatz, Equation~\eqref{eqn:Delta.t.2D}, i.e.~$2t_{,zw} = \frac13\,S_{111}S_{222}$, rewrites as
\[
	\mathsf{s}_1\mathsf{s}_2 = \mathsf{d}_z\mathsf{d}_w-\mathsf{d}\mathsf{d}_{zw}\,,
\] 
and we observe that it coincides with the Plücker relation (3.11a) in \cite{Kress&Schoebel}.
Next, we insert~\eqref{eqn:Kress.Schoebel.Wilczynski} into~\eqref{eqn:tij.2D}, \eqref{eqn:Delta.t.2D} and~\eqref{eqn:DS.2D}.
In this way we confirm (i)--(v) in \cite{Kress&Schoebel}:
\begin{itemize}
	\item Equation~\eqref{eqn:tij.2D} corresponds to the equations
	\[
	(i)\quad C_{11,z} = C_{12}C_{22} + (C_{11})^2 - C_{12,w}\,,
	\qquad
	(ii)\quad C_{22,w} = C_{21}C_{11} + (C_{22})^2 - C_{21,z}\,.
	\]
	%
	\item Equation~\eqref{eqn:Delta.t.2D} corresponds to the equations
	\[
	(iii)\quad C_{11,w} = C_{12}C_{21} = C_{22,z}\,.
	\]
	\item Equation~\eqref{eqn:DS.2D} corresponds to
	\[
	(iv)\quad C_{12,z} = C_{12}C_{11}\,,\qquad
	(v)\quad C_{21,w} = C_{21}C_{22}\,.
	\]
\end{itemize}
In~\cite[(3.7)]{Kress&Schoebel}, we find the relations
\[
 A_{zz} = 0\,,\quad
 A_{zw} = D_{zz}\,,\quad
 B_{ww} = 0\,,\quad
 B_{wz} = D_{ww}\,.
\]
These can be identified with, respectively, the equations $(iv)$, $(i)$, $(v)$ and $(ii)$. We have also already encountered these in Proposition~\ref{prop:holomorphic.A} where we discussed the holomorphic function $\beta$.
Subsequently, we establish the correspondence
\[
 \frac{D_{,i}}{D} = -\frac43\,t_i = \ln(\mathsf{d})_{,i}
\]
and obtain the identity $D = \mathsf{d}$, as the integration constant is irrelevant.
For the remaining derivatives we have
\begin{align*}
 C_{211z} &= C_{211}C_{11} +C_{11}^2C_{21} -2C_{21}C_{122} +2C_{12}C_{21}C_{22} \\
 C_{211w} &= C_{211}C_{22} + C_{21}^2C_{12} \\
 C_{122z} &= C_{122}C_{11} + C_{12}^2C_{21} \\
 C_{122w} &= C_{122}C_{22} +C_{22}^2C_{12} -2C_{12}C_{211} +2C_{21}C_{12}C_{11}
\end{align*}
which correspond to~\eqref{eqn:DXi.2D}.
We have thus established the following correspondences between the equations~\eqref{eqn:NLP.abstract} and equations in \cite{Kress&Schoebel,Kalnins&Kress&Miller,KKM07b}, namely
\begin{align*}
 D &= \exp\left(-\frac43\,t\right) \\
 A_z &= \frac32\,D\,S_{111} = \frac32\,S_{111}\,e^{-\frac43\,t} \\
 B_w &= \frac32\,D\,S_{222} = \frac32\,S_{222}\,e^{-\frac43\,t} \,.
\end{align*}
Employing Proposition~\ref{prop:holomorphic.A} and the considerations after it, we infer
\begin{equation}\label{eq:AB}
	A_{zz} = 0\,,\quad
	A_{zw} = D_{zz}\,,\qquad
	B_{zw} = D_{ww}\,,\quad	
	B_{ww}=0\,,
\end{equation}
which are precisely the equations (3.7) of \cite{Kress&Schoebel}. 
Summarising further, we similarly have the correspondences
\begin{align*}
	C_{11} &= \frac43t_1 & C_{12} &= \frac43 s_1
	& C_{112} &= \frac23\xi_2+\frac89s_2t_1-\frac83\frac{\phi_{\bar z}}{\phi}t_{\bar z}
	\\
	C_{21} &= \frac43s_2 & C_{22} &= \frac43 t_{\bar z}
	& C_{122} &= \frac23\xi_1+\frac89s_1t_{\bar z}-\frac83\frac{\phi_{z}}{\phi}t_{z}
\end{align*}
between the terminology employed in the present paper and that in \cite{Kress&Schoebel,Kalnins&Kress&Miller,KKM07b}.
From~\eqref{eq:AB} we immediately infer two further equations,
\[
	D_{zzz}=0\,,\quad
	D_{www}=0\,
\]
which indeed correspond to~\eqref{eqn:Remn} in the case $R=0$. We have encountered them already after Proposition~\ref{prop:holomorphic.A}.
A corresponding pair of equations can also be formulated in terms of $C_{ij}$, see in \cite{Kress&Schoebel,Kalnins&Kress&Miller,KKM07b}.
Comparing with these existing works, we observe that one further condition exists, namely
\[
	D_{zzww} = 0\,,
\]
and this equation indeed is~\eqref{eqn:DRemn} with $R=0$. This can similarly be phrased in terms of a quartic equation in $C_{ij}$, as shown in \cite{Kress&Schoebel,KKM07b}. A direct computation shows that, with $R=0$, any subsequent differential consequence of~\eqref{eqn:DRemn} is already algebraically generated by~\eqref{eqn:NLP.abstract}, \eqref{eqn:Remn} and~\eqref{eqn:DRemn} itself.
We have therefore confirmed consistency with~\cite{Kress&Schoebel,KKM07b} and seen how the three different terminologies compare.

\section{The two-sphere}
\label{sec:sphere}

In this final section, we consider the unit 2-sphere $\mathds
S^2\subset\mathbb R^3$ with the round metric $g$, induced from the standard
scalar product $\hat g$ on $\mathbb R^3$.  It has scalar curvature $R=2$, such
that \eqref{eqn:Remn} reads
\begin{equation}
	\frac23\,\Xi_{ia}t^a + \frac49\,S_{iab}t^at^b + \frac59\,S_{abc}S^{abc}t_i - S_{iab}\Xi^{ab}
	= -2t_i\,.
	\label{eqn:Remn.sphere}
\end{equation}
Note that this is a set of algebraic equations on each tangent space.  In what
follows we will show how to transform this into an explicit set of purely algebraic
equations.  The latter then equip the set of superintegrable systems on
$\mathds S^2$ with the structure of an algebraic variety, invariant under a
linear action of the isometry group $\operatorname{SO}(3)$.

A \emph{special conformal Killing tensor} is a symmetric tensor $L_{ij}$
satisfying Sinyukov's equation,
\begin{subequations}
	\begin{equation}
		L_{ij,k}=\lambda_ig_{jk}+\lambda_jg_{ik},
		\label{eq:Sinyukov}
	\end{equation}
	where
	\begin{equation}
		\lambda_k=\tfrac12\nabla_k\tr L,
	\end{equation}
	as can be seen from contracting $i$ and $j$ in \eqref{eq:Sinyukov}.
\end{subequations}
Sinyukov's equation can be extended to a prolongation, expressing the first
(and hence all) derivatives of $L_{ij}$ and $\lambda_k$ linearly in $L_{ij}$
and $\lambda_k$ \cite{Sinyukov_1994,MS1983,Mikes1998}.  For spheres this
prolongation simplifies to
\begin{align}
	L_{ij,k}&=\lambda_ig_{jk}+\lambda_jg_{ik}\\
	\lambda_{i,j}&=-L_{ij}-2\lambda g_{ij}.
	\label{eq:Sinyukov:prolongation}
\end{align}
Since $\lambda_k$ is determined by $L_{ij}$, these equations imply that the
maximal dimension of the space of special conformal Killing tensors on
$\mathds S^n$ is $n(n+1)/2$.  On the other hand, it is not difficult to check
that the restriction of a constant symmetric tensor $\hat L$ on $\mathbb
R^{n+1}$ to the sphere $\mathds S^n\subset\mathbb R^{n+1}$ is a special
conformal Killing tensor.  Explicitly, we have, for	$x\in\mathds S^n$ and
$v,w\in T_xS^n$ considered as vectors in $x,v,w\in\mathbb R^{n+1}$ with
$\hat g(x,x)=1$ and $v,w\perp x$,
\begin{align}
	L_x(v,w)&=\hat L(v,w)&
	\lambda_x(v)&=-\hat L(v,x)
	\label{eq:restriction}
\end{align}
As this restriction is injective, we see from its dimension that the space of
special conformal Killing tensors on $\mathds S^n$ is isomorphic to
$S^2\mathbb R^{n+1}$.  Note that the metric $g$ on $\mathds S^n$ is a special
Killing tensor, induced from $\hat g$ on $\mathbb R^{n+1}$.

Every special conformal Killing tensor $L$ gives rise to a proper Killing
tensor
\begin{equation}
	K=L-(\tr L)g,
	\label{eq:L2K}
\end{equation}
which we call a \emph{special Killing tensor}.  Since $\tr K=(1-n)\tr L$, this
defines a bijection between the space of special conformal Killing tensors and
that of special Killing tensors.  In dimension $n=2$, the latter has the same
dimension as the space of all Killing tensors.  Hence on $\mathds S^2$ every
Killing tensor $K$ arises from a special conformal Killing tensor $L$ via
\eqref{eq:L2K} and thus from the restriction of a constant tensor $\hat L\in
S^2\mathbb R^3$.

From Section~\ref{sec:preliminaries} we know that the space of Killing tensors
in a non-degenerate superintegrable system determines its structure tensor and
hence the entire superintegrable system.  In two dimensions, this space is
spanned by three linearly independent Killing tensors $K^{(0)}$, $K^{(1)}$ and
$K^{(2)}$, given by $\hat L^{(0)},\hat L^{(1)},\hat L^{(2)}\in S^2\mathbb
R^3$.  We can choose one of them to be the metric, i.e. $K^{(0)}=g$ and $\hat
L^{(0)}=\hat g$, and therefore assume that $\hat L^{(1)}$ and $\hat L^{(2)}$
be trace-free.

To summarise, a superintegrable system on $\mathds S^2$ is completely
determined by a two-dimensional subspace $\langle\hat L^{(1)},\hat
L^{(2)}\rangle$ in the space $S^2_0\mathbb R^3$ of trace-free symmetric
tensors on $\mathbb R^3$ or, equivalently, a point in the Grassmannian
$G_2(S_0\mathbb R^3)$.  Under the Plücker embedding $G_2(S_0\mathbb
R^3)\hookrightarrow\mathbb P\Lambda^2S_0\mathbb R^3$, this is mapped to the
tensor $\hat P:=\hat L^{(1)}\wedge\hat L^{(2)}$.  Its entries are the Plücker
coordinates of the subspace $\langle\hat L^{(1)},\hat L^{(2)}\rangle$, given by
\begin{equation}
	\hat P_{ijkl}
	=\hat L^{(1)}_{ij}\hat L^{(2)}_{kl}-\hat L^{(2)}_{ij}\hat L^{(1)}_{kl}.
	\label{eq:Plücker:Phat}
\end{equation}
Note that $\dim S_0\mathbb R^3=5$, so that each superintegrable system on
$\mathds S^2$ can be uniquely encoded in a rank one skew symmetric
$5\!\times\!5$\,-matrix.  This can be made explicit in a local basis
$(\partial_z,\partial_w)$ by decomposing $\hat L$ under the decomposition
$T_x\mathbb R^3=T_x\mathds S^n\oplus x$ using \eqref{eq:restriction},
\[
	\hat L=
	\begin{pmatrix}
		-L_{zz}-L_{ww}&-\lambda_z&-\lambda_w\\
		-\lambda_z&L_{zz}&L_{zw}\\
		-\lambda_w&L_{zw}&L_{ww}
	\end{pmatrix},
\]
and assembling the components of this tensor into a vector
\[
	l
	=(\lambda_w,L_{ww},L_{wz},L_{zz},\lambda_z)
	=:(l_{-2},l_{-1},l_0,l_1,l_2).
\]
Then the above mentioned $5\!\times\!5$\,-matrix is given by
\[
	p=l^{(1)}\wedge l^{(2)}
\]
with entries
\begin{align}
	p_{ 0, 1}&=L_{zw}^{(1)}L^{(2)}_{zz}-L_{zw}^{(2)}L^{(1)}_{zz}&
	p_{ 0,-1}&=L_{zw}^{(1)}L^{(2)}_{ww}-L_{zw}^{(2)} L^{(1)}_{ww}\notag\\
	p_{ 0, 2}&=L_{zw}^{(1)}\lambda^{(2)}_z-L_{zw}^{(2)}\lambda^{(1)}_z&
	p_{ 0,-2}&=L_{zw}^{(1)}\lambda^{(2)}_w-L_{zw}^{(2)}\lambda^{(1)}_w\notag\\
	p_{ 1, 2}&=L^{(1)}_{zz}\lambda^{(2)}_z-L^{(2)}_{zz}\lambda^{(1)}_z&
	p_{-1,-2}&=L^{(1)}_{ww}\lambda^{(2)}_w-L^{(2)}_{ww}\lambda^{(1)}_w\label{eq:Plücker:pij}\\
	p_{ 1,-2}&=L^{(1)}_{zz}\lambda^{(2)}_w-L^{(2)}_{zz}\lambda^{(1)}_w&
	p_{-1, 2}&=L^{(1)}_{ww}\lambda^{(2)}_z-L^{(2)}_{ww}\lambda^{(1)}_z\notag\\
	p_{ 1,-1}&=L^{(1)}_{zz}L^{(2)}_{ww}-L^{(2)}_{zz}L^{(1)}_{ww}&
	p_{ 2,-2}&=\lambda^{(1)}_z\lambda^{(2)}_w-\lambda^{(2)}_z\lambda^{(1)}_w.\notag
\end{align}
These functions are the components of the three tensors
\begin{align*}
	P_{ijkl}&=L^{(1)}_{ij}L^{(2)}_{kl}-L^{(2)}_{ij}L^{(1)}_{kl}\\
	P_{ijk}&=L^{(1)}_{ij}\lambda^{(2)}_k-L^{(2)}_{ij}\lambda^{(1)}_k\\
	P_{ij}&=\lambda^{(1)}_i\lambda^{(2)}_j-\lambda^{(1)}_i\lambda^{(2)}_j.
\end{align*}
For instance,
\begin{align*}
	p_{ 1, 2}&=P_{zzz}&
	p_{ 1,-2}&=P_{zzw}\\
	p_{-1, 2}&=P_{wwz}&
	p_{-1,-2}&=P_{www}.
\end{align*}
Note that due to \eqref{eq:Sinyukov:prolongation} the covariant derivatives of
$P_{ijkl}$, $P_{ijk}$ and $P_{ij}$ can be expressed linearly in $P_{ijkl}$,
$P_{ijk}$ and $P_{ij}$.  For example,
\begin{align*}
	P_{ijk,l}
	&=2\young(1,2)\left(L^{(1)}_{ij,l}\lambda^{(2)}_{k}+L^{(1)}_{ij}\lambda^{(2)}_{k,l}\right)\\
	&=2\young(1,2)\left(
		\lambda^{(1)}_i\lambda^{(2)}_kg_{jl}+\lambda^{(1)}_j\lambda^{(2)}_kg_{il}
		-L^{(1)}_{ij}L^{(2)}_{kl}-2L^{(1)}_{ij}\lambda^{(2)}g_{kl}
	\right)\\
	&=P_{ik}g_{jl}+P_{jk}g_{il}-P_{ijkl}-P\indices{_{ij}^a_a}g_{kl}.
\end{align*}
In order to translate Equation~\eqref{eqn:Remn.sphere} to purely algebraic
equations, we rewrite the Bertrand-Darboux condition~\eqref{eq:dKdV:ij} for a
special Killing tensor~\eqref{eq:L2K} in terms of the corresponding special
conformal Killing tensor $L$ as
\[
	d(LdV)=d\lambda\wedge dV
\]
or, in coordinates,
\begin{equation}
	\young(j,k)\left(
		L\indices{^a_j}V_{,ka}
		+L\indices{^a_{j,k}}V_{,a}
		+\lambda_jV_{,k}
	\right)=0,
\end{equation}
Using the Sinyukov Equation~\eqref{eq:Sinyukov}, this simplifies to
\begin{equation}
	\young(j,k)\left(
		L\indices{^a_j}V_{,ka}
		+3\lambda_jV_{,k}
	\right)=0.
\end{equation}
Writing this equation for both, $L^{(1)}$ and $L^{(2)}$, in complex local
coordinates $z$ and $w=\bar z$, we obtain
\[
	\begin{aligned}
		\phi^{-2}\bigl(L^{(1)}_{ww}V_{,zz}-L^{(1)}_{zz}V_{,ww}\bigr)&=3\bigl(\lambda^{(1)}_zV_{,w}-\lambda^{(1)}_wV_{,z}\bigr)\\
		\phi^{-2}\bigl(L^{(2)}_{ww}V_{,zz}-L^{(2)}_{zz}V_{,ww}\bigr)&=3\bigl(\lambda^{(2)}_zV_{,w}-\lambda^{(2)}_wV_{,z}\bigr).
	\end{aligned}
\]
Solving this linear system for $V_{,zz}$ and $V_{,ww}$, we obtain
\begin{equation}
	\begin{pmatrix}
		V_{,zz}\\
		V_{,ww}
	\end{pmatrix}
	=
	3\,\frac{\phi^2}{p_{1,-1}}\,
	\begin{pmatrix}
		p_{-2,1} & p_{1,2} \\
		p_{-2,-1} & p_{-1,2}
	\end{pmatrix}
	\begin{pmatrix}
		V_{,z}\\
		V_{,w}
	\end{pmatrix}.
	\label{eq:V'2V'':pij}
\end{equation}
From this we can now directly read off the components of the
structure tensor:
\begin{align*}
	t_z&=3\phi^2\frac{p_{-2,1}}{p_{1,-1}}&
	t_{\bar z}&=\overline{t_z}&
	s&=3\phi^2\frac{p_{1,2}}{p_{1,-1}}&
\end{align*}
In a stereographic projection $\mathds S^2\setminus\{\text{North
Pole}\}\to\mathbb C$, given by mapping $x=(x_1,x_2,x_3)$ to
\[
	z=\frac{x_1+ix_2}{1-x_3},
\]
the conformal factor is
\[
	\phi=\frac2{1+|z|^2}=1-x_3.
\]
Therefore the components of the structure tensor are rational expressions in
the functions \eqref{eq:Plücker:pij} and the components of
$x=(x_1,x_2,x_3)\in\mathds S^2$.  The same is true for the components of
\[
	\Xi
	=\xi dz^2+\bar\xi d\bar z^2
	=\div(S)+\frac23\,S(dt),
\]
which are determined by
\[
	\xi=2\left(\partial_{\bar z}-\phi z+\tfrac23t_{\bar z}\right)s.
\]
This is true, because
\begin{enumerate}
	\item The $p_{ij}$ are components of the tensors $P_{ijkl}$, $P_{ijk}$ and $P_{ij}$.
	\item The covariant derivatives of $P_{ijkl}$, $P_{ijk}$ and $P_{ij}$ are linear in $P_{ijkl}$, $P_{ijk}$ and $P_{ij}$.
	\item
		The nonzero Christoffel symbols have the form $\pm\phi z$ or $\pm\phi
		\bar z$ and are therefore rational in $z$ and $\bar z$.
\end{enumerate}
Since the stereographic projection is a rational map,
Equation~\eqref{eqn:Remn.sphere} is rational in the functions
\eqref{eq:Plücker:pij} and the components of $x\in\mathds S^2$.  Multiplying
by the least common denominator, we obtain polynomial equations involving $x$
and arbitrary tangent vectors.  Decomposing these polynomials into irreducible
components under the isometry group finally yields algebraic equations
in the Plücker coordinates \eqref{eq:Plücker:Phat}, independent of $x$ and any
tangent vectors.  These define the algebraic variety parametrising all
non-degenerate second order superintegrable systems on the 2-sphere $\mathds
S^2$.

Solving the algebraic equations derived above is outside the scope of the
present paper and will be subject to future publications.  Note that in
solving the condition~\eqref{eqn:Remn.sphere} by translating it to purely
algebraic equations, all its differential consequences are automatically
satisfied.  In the next step, the corresponding ideal can be studied via
algebraic-geometric methods similar to those in~\cite{Kress&Schoebel}, where
the Euclidean plane was investigated in a similar way, leading to the well
established list for second-order superintegrable systems on surfaces of
constant curvature in~\cite{Kalnins&Kress&Pogosyan&Miller}. In this way, the
parameter space is given the structure of an algebraic variety endowed with a
linear isometry action. Among other insights, this revealed a previously
unknown combinatoral structure in the guise of line arrangements, which label
isometry types in the hierarchy of flat superintegrable systems.  Similar
results are to be expected for the 2-sphere applying the methodology developed
here.

\bibliographystyle{amsalpha}
\bibliography{refs}
\medskip

\end{document}